\documentclass[reqno]{amsart}
\usepackage{amsthm}
\usepackage{times,amssymb,amsmath,amsfonts,bbm,mathrsfs,color,nicefrac,graphicx,enumitem,booktabs}
\usepackage[mathscr]{eucal}
\usepackage[top=1.5in,bottom=1.43in,left=1.25in,right=1.25in]{geometry}
\allowdisplaybreaks
\usepackage{silence}
\usepackage{hyperref} 
\usepackage[font={it}]{caption}
\usepackage[margin=0.05in]{subcaption}
\usepackage[bottom]{footmisc}

\usepackage[normalem]{ulem}
\newcommand\nc\newcommand
\nc\bfa{{\boldsymbol a}}\nc\bfA{{\boldsymbol A}}\nc\cA{{\mathscr A}}
\nc\bfb{{\boldsymbol b}}\nc\bfB{{\boldsymbol B}}\nc\cB{{\mathscr B}}
\nc\bfc{{\boldsymbol c}}\nc\bfC{{\boldsymbol C}}\nc\cC{{\mathscr C}}
\nc\bfd{{\boldsymbol d}}\nc\bfD{{\boldsymbol D}}\nc\cD{{\mathscr D}}
\nc\bfe{{\boldsymbol e}}\nc\bfE{{\boldsymbol E}}\nc\cE{{\mathscr E}}
\nc\bff{{\boldsymbol f}}\nc\bfF{{\boldsymbol F}}\nc\cF{{\mathscr F}}
\nc\bfg{{\boldsymbol g}}\nc\bfG{{\boldsymbol G}}\nc\cG{{\mathscr G}}
\nc\bfh{{\boldsymbol h}}\nc\bfH{{\boldsymbol H}}\nc\cH{{\mathscr H}}
\nc\bfi{{\boldsymbol i}}\nc\bfI{{\boldsymbol I}}\nc\cI{{\mathcal I}}
\nc\bfj{{\boldsymbol j}}\nc\bfJ{{\boldsymbol J}}\nc\cJ{{\mathscr J}}
\nc\bfk{{\boldsymbol k}}\nc\bfK{{\boldsymbol K}}\nc\cK{{\mathscr K}}
\nc\bfl{{\boldsymbol l}}\nc\bfL{{\boldsymbol L}}\nc\cL{{\mathscr L}}
\nc\bfm{{\boldsymbol m}}\nc\bfM{{\boldsymbol M}}\nc{\cM}{{\mathscr M}}
\nc\bfn{{\boldsymbol n}}\nc\bfN{{\boldsymbol N}}\nc\cN{{\mathscr N}}
\nc\bfo{{\boldsymbol o}}\nc\bfO{{\boldsymbol O}}\nc\cO{{\mathscr O}}
\nc\bfp{{\boldsymbol p}}\nc\bfP{{\boldsymbol P}}\nc\cP{{\mathscr P}}\nc\eP{{\EuScriptP}}\nc\fP{{\mathfrak P}}
\nc\bfq{{\boldsymbol q}}\nc\bfQ{{\boldsymbol Q}}\nc\cQ{{\mathscr Q}}
\nc\bfr{{\boldsymbol r}}\nc\bfR{{\boldsymbol R}}\nc\cR{{\mathscr R}}
\nc\bfs{{\boldsymbol s}}\nc\bfS{{\boldsymbol S}}\nc\cS{{\mathscr S}}
\nc\bft{{\boldsymbol t}}\nc\bfT{{\boldsymbol T}}\nc\cT{{\mathscr T}}
\nc\bfu{{\boldsymbol u}}\nc\bfU{{\boldsymbol U}}\nc\cU{{\mathscr U}}
\nc\bfv{{\boldsymbol v}}\nc\bfV{{\boldsymbol V}}\nc\cV{{\mathscr V}}
\nc\bfw{{\boldsymbol w}}\nc\bfW{{\boldsymbol W}}\nc\cW{{\mathscr W}}
\nc\bfx{{\boldsymbol x}}\nc\bfX{{\boldsymbol X}}\nc\cX{{\mathscr X}}
\nc\bfy{{\boldsymbol y}}\nc\bfY{{\boldsymbol Y}}\nc\cY{{\mathscr Y}}
\nc\bfz{{\boldsymbol z}}\nc\bfZ{{\boldsymbol Z}}\nc\cZ{{\mathscr Z}}
\nc{\bb}{{\mathbbm{1}}}
\nc\reals{{\mathbb R}}
\nc{\half}{{\nicefrac12}}

\newtheorem{theorem}{Theorem}[section]
\newtheorem{lemma}[theorem]{Lemma}

\newtheorem{proposition}[theorem]{Proposition}

\newtheorem{corollary}[theorem]{Corollary}

\theoremstyle{remark}

\DeclareMathOperator{\diam}{diam}

\renewcommand{\hat}{\widehat}

\begin{document}
		\title{Stolarsky's invariance principle for finite metric spaces}
		\author{Alexander Barg}\address{Department of ECE and Institute for Systems Research, University of Maryland, College Park, MD 20742, USA and Inst. for Probl. Inform. Trans., Moscow, Russia}\email{abarg@umd.edu}

 \begin{abstract}  
Stolarsky's invariance principle quantifies the deviation of a subset of a metric space from the uniform distribution. Classically derived for spherical sets, it has been recently studied in a number of other situations, revealing a general structure behind various forms of the main identity. In this work we consider the case of finite metric spaces, relating the quadratic discrepancy of a subset to a certain function of the distribution of distances in it. Our main results are related to a concrete form of the invariance principle for the Hamming space. We derive several equivalent versions of the expression for the discrepancy of a code, including expansions of the discrepancy and associated kernels in the Krawtchouk basis. Codes that have the smallest possible quadratic discrepancy among all subsets of the same cardinality can be naturally viewed as energy minimizing subsets in the space.  Using linear programming, we find several bounds on the minimal discrepancy and give examples of minimizing configurations. In particular, we show that all binary perfect codes have the smallest possible discrepancy.
 \end{abstract}
 \maketitle
 
\renewcommand{\thefootnote}{\arabic{footnote}}
\setcounter{footnote}{0}		
		
{\renewcommand{\thefootnote}{}\footnotetext{
\vspace*{-.15in}

{This paper appears in {\em Mathematika}, vol.~{\bf 67}, no.~1, 2021, pp. 158--186, \href{https://londmathsoc.onlinelibrary.wiley.com/doi/10.1112/mtk.12066}{doi:10.1112/mtk12066}. This version corrects the statement of Theorem 6.1 and contains a new Appendix with a simple proof of Eq.~\eqref{eq:SoS}.
}}}
\renewcommand{\thefootnote}{\arabic{footnote}}
\setcounter{footnote}{0}

\section{Introduction}
 Let ${\cX}$ be a finite metric space with $\diam({\cX})=n,$ where $n$ is a positive integer. 
Suppose that the distance $d(x,y)$ takes the values $0,1,\dots,n$ and let
$B(x,t)$ be a metric ball in ${\cX}$ of radius $t$ with center $x$. We assume that ${\cX}$ is distance-invariant, so the volume of the ball $B(x,t)$
does not depend on the center.
For an $N$-point subset $Z=\{z_1,\dots,z_N\}\subset {\cX}$ define the quadratic discrepancy of $Z$ as follows:
   \begin{equation}\label{eq:D}
   D^{L_2}(Z)=\sum_{t=0}^n (D_t(Z))^2
   \end{equation}
where 
   \begin{equation}\label{eq:DG}
   D_t(Z):=\Big( \sum_{x\in {\cX}}\Big(\frac 1N \sum_{j=1}^N \mathbbm{1}_{B(x,t)}(z_j)-\frac1{|{\cX}|}|B(x,t)|\Big)^2\Big)^{1/2}.
   \end{equation}
The discrepancy measures the quadratic deviation of $Z$ from the uniform distribution in regards to the metric balls in ${\cX}.$ In this paper we
relate $D^{L_2}(Z)$ to the structure of distances in the subset $Z$ with a special attention to the Hamming space ${\cX}=\{0,1\}^n.$

A general problem of estimating quadratic discrepancy in metric spaces has a long history which developed both from the perspective
of approximation theory and geometry of the space \cite{Beck1987,Matousek1999}, with special attention devoted to subsets of the real sphere $S^d$.
The quadratic discrepancy for a configuration $Z\subset S^d$ is defined as 
  \begin{equation}\label{eq:Sd}
  D^{L_2}(Z)=\int_{-1}^1\int_{S^d}\Big(\frac 1N\sum_{i=1}^N\mathbbm{1}_{C(x,t)}(z_i)-\sigma (C(x,t))\Big)^2d\sigma(x)dt
  \end{equation}
where $C(x,t):=\{y\in S^d|(x,y)\ge t\}$ is a spherical cap with center at $x\in S^d$ and $d\sigma(x)$
is the normalized surface measure. Several recent works studied discrepancy of finite point sets on $S^d$
and related homogeneous spaces, e.g., the real and complex projective spaces \cite{Bilyk2018,Brauchart2013,Skriganov2020,Skriganov2019}.
These studies revolve around a unifying topic that originates in Stolarsky's works \cite{Stolarsky1973,Stolarsky1975} and is related to the following remarkable identity, nowadays called Stolarsky's invariance principle:
for a finite subset $Z\subset S^d$ 
   \begin{equation}\label{eq:Ssph}
   D^{L_2}(Z)=C_d(\langle \|x-y\|\rangle_{S^d}-\langle\|x-y\|\rangle_Z),
   \end{equation}
where $\langle\cdot\rangle$ refers to the average value of the argument over the subscript set, and $C_d$ is a universal constant that depends
only on the dimension $d$. This relation enables one, among other things, to establish universal bounds on discrepancy for various classes of configurations \cite{Bilyk2018,Skriganov2017,Skriganov2019}, and affords a number of generalizations. Among them,
invariance for $D^{L_2}(Z)$ defined with respect to the geodesic distance on $S^d$, or with respect to other subsets of $S^d$, or for a continuous analog of the discrepancy \cite{Bilyk2018}. The invariance principle also connects $D^{L_2}(Z)$ to
a classical problem of Fejes T{\'o}th of maximizing the sum of distances over finite sets of a given cardinality \cite{Bilyk2019}. The cited
papers also offer insightful general discussions of, as well as many more references for, the invariance principle; of them we mention \cite{Bilyk2018,Skriganov2019} as our main motivation.

In this paper we study a discrete version of the invariance principle, modeling our definition of quadratic discrepancy \eqref{eq:D} on the
classic definition \eqref{eq:Sd}. We begin with general finite metric spaces, deriving a version of the relation \eqref{eq:Ssph}. It turns out that the quadratic discrepancy of a subset $Z$ of a finite metric space ${\cX}$ is conveniently expressed via a function on ${\cX}\times {\cX}$ defined
as 
   \begin{equation}\label{eq:lambda0}
     \lambda(x,y)=\frac 12\sum\nolimits_{u\in {\cX}}|d(x,u)-d(y,u)|.
   \end{equation}
A general form of the Stolarsky invariance principle, proved in Sec.~\ref{sec:Stolarsky}, asserts that $D^{L_2}(Z)$ equals the difference between the average value of $\lambda$ over the entire space and its average over $Z.$ This result is implicit in earlier works such as \cite{Brauchart2013,Bilyk2018,Skriganov2019}, and we comment on their results in the main text of the paper.

Our main results are related to the case of the Hamming space ${\cX}=\{0,1\}^n$. In Sec.~\ref{sec:Hamming} we give several equivalent, but different-looking expressions for $\lambda,$ showing that $\lambda(x,y)$ is expressed in terms of the central binomial coefficient $\binom w{w/2}$ (after accounting 
for integrality constraints), where $w=d(x,y).$ This result enables us to connect the discrepancy of a binary code with 
the structure of distances in it, and to find exact expressions or bounds for discrepancy of several classes of binary codes. 

Analyzing discrepancy of finite sets in metric spaces is often facilitated by considering expansions of $D^{L_2}$ and associated kernels into 
series of spherical functions. Such expansions were studied for the sphere $S^d$ and related projective spaces in \cite{BilykDai2019,Skriganov2019,Skriganov2020}, where the spherical functions are given by certain Jacobi polynomials. The authors of the cited works used estimates of the coefficients in the expansions to derive bounds on discrepancy of finite point configurations. Paper \cite{Skriganov2019} also showed that optimal spherical designs have asymptotically the smallest possible discrepancy among sets of their cardinality. Related ideas on a connection between spherical designs and asymptotically uniformly distributed spherical sets were addressed in
\cite{Brauchart2019}. In order to derive estimates of discrepancy for binary codes of a given cardinality, in Sec.~\ref{sec:Fourier} we derive Fourier-Krawtchouk expansions of the kernel $\lambda(\cdot)$ and the discrepancy $D^{L_2}(Z).$

The discrepancy kernel forms an example of a potential function on ${\cX}$, and thus, minimizing the value $D^{L_2}(Z)$ over codes $Z$ of a
given size can be addressed via the linear programming approach to energy minimization. This line of research started with the work
\cite{Yudin1993} and has enjoyed considerable attention in recent years. Most works on energy-minimizing configurations address the setting of 
finite point sets on the sphere in $\reals^d$ and related homogeneous spaces; see a comprehensive recent book \cite{BHS2019} for an extensive 
overview as well as numerous references. Universal bounds on energy of spherical codes \cite{Cohn2007,BDHSS2016,BDHDSS2019} can be obtained for the class of {\em completely monototic} potentials such as the Euclidean distance.

Turning to the Hamming space, linear programming bounds on energy of codes were studied in recent papers \cite{BDHSS2017,Cohn2014} as well as in 
earlier works \cite{ash99a,ash01b}. At the same time, universal bounds of \cite{BDHSS2017,BDHSS2019a} are not applicable to the problem
of discrepancy because the potential $\lambda(\cdot,\cdot)$ fails to be completely monotonic. In particular, this implies that the universally optimal codes listed in \cite{Cohn2014} are not necessarily LP-optimal for the discrepancy problem. We use linear programming together with the Krawtchouk expansion of the function $\lambda(d(x,y))$ to derive several lower bounds on $D^{L_2}(N)$. We also prove that the Hamming codes and other perfect codes have the smallest possible discrepancy among all codes of their cardinality.

By combining random choice and linear programming estimates, 
we prove the following bounds on the minimum discrepancy of binary codes of length $n$ and size $N$. Define the extremal discrepancy as
  \begin{equation}\label{eq:Dmin}
  D^{L_2}(n,N):=\min_{Z\subset\{0,1\}^n, |Z|=N} D^{L_2}(Z),
  \end{equation}
then we show that 
\begin{equation*}
    c\frac1{\sqrt{n}}\frac {2^n}{N}\le D^{L_2}(n,N)\le C\sqrt{n}\frac{2^n}{N},
\end{equation*}
where $c,C$ are some constants that depend only on $n$; see Theorem \ref{prop:bounds} below.

\section{Stolarsky's invariance for a finite metric space}\label{sec:Stolarsky}

The quadratic discrepancy of $Z\subset {\cX}$ can be expressed as the difference between the average values of $\lambda$ over the entire space ${\cX}$ and the subset $Z$ as follows.
\begin{theorem}[{\sc Stolarsky's invariance principle}]\label{thm:SG} Let $Z=\{z_1,\dots,z_N\}$ be a subset of a finite metric space ${\cX}.$ Then
   \begin{equation}\label{eq:SG}
   D^{L_2}(Z)=\frac12\Big(\frac1{|{\cX}|^2}\sum_{x,y\in{\cX}}\sum_{u\in{\cX}}|d(x,u)-d(y,u)|
      -\frac 1{N^2}\sum_{i,j=1}^N\sum_{u\in{\cX}}|d(z_i,u)-d(z_j,u)|\Big).
   \end{equation}
\end{theorem}
\begin{proof}
Starting with \eqref{eq:DG}, we compute
   \begin{align}
    D_t(Z)^2&=\sum_{x\in {\cX}}\Big[\Big(\frac 1{N}\sum_{j=1}^N\mathbbm{1}_{B(x,t)}(z_j)\Big)^2-\frac 2N\sum_{j=1}^N\mathbbm{1}_{B(x,t)}(z_j)
    \frac{|B(x,t)|}{|{\cX}|}+\frac{|B(x,t)|^2}{|{\cX}|^2}\Big]\nonumber\\
    &=\frac 1{N^2}\sum_{x\in {\cX}}\Big(\sum_{j=1}^N\mathbbm{1}_{B(x,t)}(z_j)\Big)^2-\frac 2N\sum_{j=1}^N\sum_{x\in{\cX}}\mathbbm{1}_{B(z_j,t)}(x)
      \frac{|B(x,t)|}{|{\cX}|}+\frac{|B(x,t)|^2}{|{\cX}|}\nonumber\\
      &=\frac 1{N^2}\Big(\sum_{x\in{\cX}}\sum_{i,j=1}^N \mathbbm{1}_{B(x,t)}(z_i)\mathbbm{1}_{B(x,t)}(z_j)\Big)-\frac{|B(x,t)|^2}{|{\cX}|}\nonumber\\
      &=\frac1{N^2}\sum_{i,j=1}^N |B(z_i,t)\cap B(z_j,t)|-\frac{|B(u,t)|^2}{|{\cX}|}, \label{eq:cap1}
   \end{align}
where on the last line $u$ is any fixed point in ${\cX}.$

To find an expression for $D^{L_2}(Z)$ in \eqref{eq:D}, we need to sum \eqref{eq:cap1} on $t$.
Let us compute average intersection of the metric balls with centers at $x$ and $y$:
   \begin{align}
      \sum_{t=0}^n |B(x,t)\cap B(y,t)|&=\sum_{t=0}^n\sum_{z\in {\cX}} \mathbbm{1}_{B(x,t)}(z)\mathbbm{1}_{B(y,t)}(z)=
        \sum_{z\in{\cX}}\sum_{t=0}^n \mathbbm{1}_{B(x,t)}(z)\mathbbm{1}_{B(y,t)}(z)\nonumber\\
        &=\sum_{z\in{\cX}}\sum_{t=\max(d(z,x),d(z,y))}^n 1=\sum_{z\in {\cX}}(n+1-\max(d(z,x),d(z,y)))\nonumber\\
        &=|{\cX}|(n+1)-\sum_{z\in{\cX}}\max(d(z,x),d(z,y))\nonumber\\
        &=|{\cX}|(n+1)-\sum_{z\in{\cX}}\frac12(d(z,x)+d(z,y)+|d(z,x)-d(z,y)|)\nonumber\\
        &=|{\cX}|(n+1)-\sum_{z\in {\cX}}d(z,u)-\frac 12\sum_{z\in {\cX}}|d(z,x)-d(z,y)|.\label{eq:cap2}
  \end{align}

Now let us address the second term in \eqref{eq:cap1}:
   \begin{align}
   \sum_{t=0}^n|B(u,t)|^2&=\sum_{t=0}^n\sum_{u\in{\cX}}\mathbbm{1}_{B(x,t)}(u)\sum_{y\in{\cX}}\mathbbm{1}_{B(u,t)}(y)\nonumber\\
   &=\sum_{t=0}^n\sum_{y}\sum_{u}\mathbbm{1}_{B(x,t)}(u)\mathbbm{1}_{B(y,t)}(u)=\sum_y\sum_{t=0}^n|B(x,t)\cap B(y,t)|
   \nonumber\\
   &=\frac1{|{\cX}|}\sum_{x,y\in{\cX}}\sum_{t=0}^n |B(x,t)\cap B(y,t)|\label{eq:sq1}\\
   &=\frac1{|{\cX}|}\sum_{x,y\in{\cX}}\Big[|{\cX}|(n+1)-\sum_{z\in {\cX}}d(z,u)-\frac 12\sum_{z\in {\cX}}|d(z,x)-d(z,y)|\Big],\label{eq:square}
   \end{align}
where to obtain \eqref{eq:square} we used \eqref{eq:cap2}, and where $u$ on the last line is any fixed point in $\cX.$ Substituting \eqref{eq:cap2} and \eqref{eq:square} into \eqref{eq:cap1}
and rearranging, we obtain \eqref{eq:SG}.
\end{proof}

Similar proofs of Stolarsky's principle for the case of the sphere $S^d(\reals)$ were given earlier in \cite{Brauchart2013,BilykLacey2017,Bilyk2018}, see also
 \cite[Sec.6.8]{BHS2019}.
In particular, the authors of \cite{Bilyk2018} used essentially the same geometric ideas, and we adopted them here for a finite metric space. Paper
\cite{Skriganov2019} considered the case of a general distance-transitive metric space ${\cX}$ equipped with the metric of symmetric
difference. In the case of finite ${\cX}$ this metric is defined as follows:
  $$
  \theta(x,y):=\frac 12\sum_{t=0}^n|B(x,t)\bigtriangleup B(y,t)|=\sum_{t=0}^n |B(x,t)|-\sum_t|B(x,t)\cap B(y,t)|,
  $$
and this definition generalizes to an arbitrary metric measure space ${\cX}$ in an obvious way. To see that $\theta$ is indeed a 
metric, note that
  $$
  \theta(x,y)=\frac12\sum_{t}\sum_{u\in{\cX}}|\mathbbm{1}_{B(x,t)}(u)-\mathbbm{1}_{B(y,t)}(u)|,
  $$
which is the $L_1$ distance between the indicator functions of the metric balls. As observed in \cite[Eq.(1.33)]{Skriganov2019}, the quadratic discrepancy of a finite zet $Z\subset{\cX}$ equals the difference between the average value of $\theta$ over the entire space and its average value over $Z.$ The author of \cite{Skriganov2019} called this relation the $L_1$ invariance principle as opposed to the more subtle $L_2$ principle given by \eqref{eq:Ssph}.

The form of the invariance principle considered above is related to the kernel $\lambda:{\cX}\times{\cX}\to \reals$ defined in \eqref{eq:lambda0}.
Using it, we can rewrite the invariance principle \eqref{eq:SG} concisely as follows:
  \begin{align}\nonumber
  D^{L_2}(Z)&=\frac 1{2^{2n}}\sum_{x,y\in {\cX}} \lambda(x,y)-\frac 1{N^2}\sum_{i,j=1}^N\lambda(z_i,z_j)\\
    &=\langle\lambda\rangle_{{\cX}}-\langle\lambda\rangle_{Z},\label{eq:kernel}
  \end{align}   
where the quantities on the last line represent the average values of $\lambda.$ Another equivalent form of \eqref{eq:kernel} is obtained
once we define the kernel 
   $$
   \mu(x,y)=\sum_{t=0}^n\mu_t(x,y),
   $$
where
  \begin{equation}\label{eq:mut}
  \mu_t(x,y):=|B(x,t)\cap B(y,t)|=\sum_{z\in{\cX}}\mathbbm{1}_{B(x,t)}(z)\mathbbm{1}_{B(y,t)}(z).
  \end{equation}
On account of \eqref{eq:cap1}, \eqref{eq:sq1}, and \eqref{eq:D} we can rewrite the expression for the quadratic discrepancy of the
subset $Z$ as follows:
  \begin{align}\label{eq:SH3mu}
   D^{L_2}(Z)&=\langle \mu\rangle_Z-\langle \mu\rangle_{{\cX}}.
   \end{align}

Both \eqref{eq:kernel} and \eqref{eq:SH3mu} have their advantages for the calculations in the Hamming space which form our main results. Namely, 
\eqref{eq:SG} is directly related to the distances in the graph while \eqref{eq:mut} is a convolution square of a function, which facilitates
the Fourier transform approach to discrepancy. Additionally, although less importantly, while both kernels $\mu$ and $\lambda$ are radial (depend
only on the distance $d(x,y)$), the former is also positive definite, which facilitates calculations of the linear-programming bounds on discrepancy. 

 \section{Stolarsky's invariance for the Hamming space}\label{sec:Hamming}
In Sections \ref{sec:Hamming}-\ref{sec:LP}, the notation ${\cX_n}=\{0,1\}^n$ refers to the binary Hamming space. For any pair of vectors $x,y\in{\cX_n}$ let $d(x,y)$ denote the Hamming distance between them. 
As above, we let $B(x,t)$ denote the ball of radius $t,0\le t\le n$ with center at $x\in {\cX_n}$ and note that
the volume $|B(x,t)|=\sum_{i=0}^t \binom nt$ does not depend on $x$. We also note the following relations 
for future use:
  for any $x\in{\cX_n}$
  \begin{align}\label{eq:sumB}
  &\sum_{t=0}^n|B(x,t)|=\sum_{t=0}^n\sum_{i=0}^t\binom ni=\sum_{i=0}^n (i+1)\binom ni=(n+2)2^{n-1}  \\ 
 \label{eq:sqsum}
   &\sum_{t=0}^n|B(x,t)|^2=2^{2n-1}(n+2)-\frac n2\binom{2n}{n}.
   \end{align}
The first of these equalities is obvious and the second was proved in \cite{Hirsch1996}.

In this section
we derive an explicit form of the Stolarsky principle \eqref{eq:SG} for the Hamming space.
As before, let $Z$ be an $N$-element subset of ${\cX_n}$, which we call a binary code. 
In the next lemma we find an explicit expression of the kernel $\lambda(x,y)$ defined in \eqref{eq:lambda0}. It depends on $d(x,y)$ and $n$, but we suppress $n$ from the notation throughout.
\begin{lemma}\label{lemma:lambda}
  Let $x,y\in{\cX_n}$ be two points such that $d(x,y)=w$. Then
  \begin{equation}\label{eq:lambda}
    \lambda(x,y)=\lambda(w):=2^{n-w}w\binom{w-1}{\lceil\frac w2\rceil-1}, \quad w=1,\dots,n.
  \end{equation}
Writing this in another form, we have 
   \begin{gather}\label{eq:eo}
   \lambda(2i-1)=\lambda(2i)=2^{n-2i}i\binom{2i}i, \quad 1\le i\le \lfloor n/2\rfloor\\
   \frac{\lambda(2i+1)}{2i+1}=\frac{\lambda(2i)}{2i}, \quad i\ge 1, \label{eq:monotone}
   \end{gather}
and thus $\lambda(i)$ is a monotone nondecreasing function of $i$ for all $i\ge 1.$
\end{lemma}
{\em Remark:} From \eqref{eq:monotone} and $\lambda(1)=2^{n-1}$ we also obtain the following expression: for $i\ge0$
   $$
   \lambda(2i+1)=\frac{(2i+1)!!}{(2i)!!}2^{n-1},
   $$
   and thus the generating function of the numbers $\lambda(2i+1)$ is
   $$
   2^{-n+1}\sum\nolimits_{i=0}^\infty \lambda(2i+1)x^i=(1-x)^{-3/2};
   $$
see also sequence \href{http://oeis.org/A001803}{A001803} in OEIS \cite{oeis}.
\begin{proof}  Without loss of generality let $x=0.$ Let $u\in {\cX_n}$ be a point and let $j=|u\cap y|$ be the intersection (the number
of common ones) of $u$ and $y$. Let $i=|u\cap y^c|=|u|-j$ be the remaining number of ones in $u$. Then 
   $$
   \sum_{u\in{\cX_n}}|d(u,x)-d(u,y)|=\sum_{u\in{\cX_n}} |(i+j)-(w-j)-i|=\sum_{u\in{\cX_n}}|2j-w|,
   $$
where we have suppressed the dependence of $i$ and $j$ on $u.$ 

Let $w$ be odd, then 
   \begin{equation*}
   \sum_{u\in{\cX_n}}|2j-w|=\sum_{u:j\ge \lceil w/2\rceil} (2j-w)+\sum_{u:j\le \lfloor w/2\rfloor}(w-2j).
\end{equation*}
The two terms on the right are equal to each other, and thus 
  \begin{align}
  \lambda(w)=\sum_{\begin{substack}{u\in{\cX_n}\\j\ge \lceil w/2\rceil}\end{substack}} 
      (2j-w)&=\sum_{i=0}^{n-w}\sum_{j=\lceil w/2\rceil}^w(2j-w)\binom{n-w}i\binom wj\nonumber\\
     &=2^{n-w}\Big[2\sum_{j=\lceil w/2\rceil}^w j\binom wj -w\sum_{j=\lceil w/2\rceil}^w \binom wj\Big]\nonumber\\
     &=2^{n-w}\Big[2w \sum_{j=\lceil w/2\rceil}^w\binom{w-1}{j-1}-w\sum_{j=\lceil w/2\rceil}^w \binom wj\Big]\nonumber\\
     &=2^{n-w}\Big[2w\Big(2^{w-2}+\frac12\binom{w-1}{(w-1)/2}\Big)-w2^{w-1}\Big]\nonumber\\
     &=2^{n-w}w\binom{w-1}{(w-1)/2}=2^{n-w}w\binom{w-1}{\lceil\frac w2\rceil-1}.\label{eq:1st}
  \end{align}
This proves \eqref{eq:lambda} for $w$ odd.
The case of $w$ even is very similar, with only minor changes:
   \begin{align*}
   \frac 12\sum_{u\in{\cX_n}}|d(u,x)-d(u,y)|&=\frac12\Big[\sum_{\begin{substack}{u\in{\cX_n}\\j\ge \frac w2+1}\end{substack}}(2j-w)+
                 \sum_{\begin{substack}{u\in{\cX_n}\\j\ge \frac w2-1}\end{substack}}(w-2j)\Big]\\
        &=2^{n-w-1}\Big[\sum_{j=\frac w2+1}^w\binom wj(2j-w)+\sum_{j=0}^{\frac w2-1}\binom wj(w-2j)\Big]\\
        &=2^{n-w}\Big[\sum_{j=\frac w2+1}^wj\binom wj-\sum_{j=0}^{\frac w2-1}j\binom wj\Big]\\
        &=2^{n-w} w\Big[\sum_{j=\frac w2+1}^w \binom{w-1}{j-1}-\sum_{j=0}^{\frac w2-1}\binom{w-1}{j-1}\Big]\\
        &=2^{n-w}w\Big[2^{w-2}-\Big(2^{w-2}-\binom{w-1}{\frac w2-1}\Big)\Big]\\
        &=2^{n-w} w\binom {w-1}{\lceil\frac w2\rceil-1},
    \end{align*}
i.e., the same as \eqref{eq:1st}.

To prove \eqref{eq:eo}, note that $\lambda(2i)=2^{n-2i}(2i)\binom{2i-1}{i-1}$ and
   $$
   \lambda(2i-1)=2^{n-2i+1}(2i-1)\binom{2i-2}{i-1}=2^{n-2i}(2i)\binom{2i-1}{i-1}=2^{n-2i}i\binom{2i}i,
   $$
   as claimed. Finally, \eqref{eq:monotone} is computed directly from \eqref{eq:eo}.
\end{proof}

Recalling \eqref{eq:kernel}, we next aim to compute the average value $\langle \lambda \rangle_{{\cX_n}}=2^{-2n}\sum_{x,y\in{\cX_n}}\lambda(d(x,y))$.
For a fixed $x$ there are $\binom nw$ vectors $y$ such that $d(x,y)=w,$ which we can use together with the expression for $\lambda$ \eqref{eq:lambda}.
Somewhat surprisingly, the resulting sum has a closed-form expression.
Namely, for any $n\ge 1$ we have:
     \begin{align}
   2^{-n}\sum_{x,y\in{\cX_n}}&\lambda(d(x,y))=\sum_{w=0}^n \binom nw\lambda(w)=\sum_{w=1}^n 2^{n-w}w\binom nw\binom{w-1}{\lceil \frac w2\rceil-1}=\frac n2\binom{2n} n.\label{eq:closed}
     \end{align}
To prove this, let us write \eqref{eq:square} for the Hamming space:
   $$
   \sum_{t=0}^n\Big(\sum_{i=0}^t \binom ni\Big)^2 = 2^{2n-1}(n+2)-2^{-n}\sum_{x,y\in {\cX_n}}\lambda(x,y).
   $$
Now from \eqref{eq:sqsum} and \eqref{eq:lambda} we find that \eqref{eq:closed} is true for all $n$.

\vspace*{.05in}{\em Remarks:} 

1. An identity related to \eqref{eq:closed} is the following:
   $
   \sum_{w=1}^n (-1)^{w+1}\binom nw\lambda(w)=\binom{2(n-1)}{n-1}.
   $
   
2. The numbers on either side of \eqref{eq:closed} as a function of $n$ form sequence
\href{http://oeis.org/A002457}{A002457} in OEIS.

  
\vspace*{.05in}

Rephrasing \eqref{eq:closed}, we obtain the following proposition.
\begin{proposition} 
  The average value of the kernel $\lambda(x,y)$ over the entire space ${\cX_n}$ equals
   \begin{equation}\label{eq:av}
   \Lambda_n:=\frac n{2^{n+1}}\binom{2n}{n}.
   \end{equation}
\end{proposition}
Thus, the average value $\Lambda_n\approx \frac n{2^{n+1}}2^{2n}/\sqrt{\pi n}=\sqrt{n/\pi}2^{n-1},$ and it increases roughly by a factor of 2 as the dimension $n$ increases by one. 

Using \eqref{eq:closed}, we obtain a simplified form of the invariance principle \eqref{eq:kernel} for the Hamming space. Namely, 
  \begin{align}\label{eq:SH2}
      D^{L_2}(Z)&=\Lambda_n-\frac1{N^2}
       \sum_{i,j=1}^N \lambda(d(z_i,z_j)).
         \end{align}
Let $$A_w:=\frac 1N|\{(z_1,z_2)\in Z^2|d(z_1,z_2)=w\}|$$ be the number of ordered pairs of elements in $Z$ at distance $w$. The set of numbers
$A(Z)=\{A_0=1,A_1,\dots,A_n\}$ is called the distance distribution of the code $Z.$ Using this concept, we can write the expression
for discrepancy in final form.
\begin{theorem}[{\sc Stolarsky's invariance for the Hamming space}]\label{thm:SHF}
Let $Z\subset \{0,1\}^n$ be a subset of size $N$ with distance distribution $A(Z).$ Then
   \begin{align}\label{eq:SH3}
   D^{L_2}(Z)&=\Lambda_n-\frac{1}{N}\sum_{w=1}^n A_w\lambda(w)
   \end{align}
\end{theorem}   

Estimating the central binomial coefficient, we can approximate $\lambda(w)\approx C 2^n \sqrt{w}$ for some constant $C<1,$ and thus
   $$
   D^{L_2}(Z)\approx \Lambda_n - C \frac{2^n}{N} \sum_{w=1}^n A_w\sqrt{w}.
   $$
A topic that we discuss in more detail below is finding codes $Z$ that have the smallest possible discrepancy among all codes of their cardinality.
An obvious observation from \eqref{eq:SH3} is that among all subsets $\{x,y\}$ of size 2 the smallest discrepancy is attained when $d(x,y)=n.$ 
Indeed, we have $D^{L_2}(Z)\ge \Lambda_n-(1/2)\lambda(d(x,y)),$ and the claim follows since $\lambda(i)$ is monotone increasing as a function of $i$ \eqref{eq:monotone}. 

Using \eqref{eq:SH3mu}, we can express $D^{L_2}(Z)$ in an equivalent form. Namely, from \eqref{eq:cap2} 
the average value of the kernel $\mu$ over ${\cX_n}$ equals 
  $$
  \langle\mu\rangle_{{\cX_n}}=2^{n-1}(n+2)-\Lambda_n,
  $$
and thus
 \begin{align*}
      D^{L_2}(Z)=\frac1{N}\sum_{w=1}^n A_w\mu(w)+\Lambda_n-2^{n-1}(n+2).\nonumber
  \end{align*}

\subsubsection*{{\sc Random codes}} To get a feeling of the possible values of $D^{L_2}(Z),$ 
let us compute the expected discrepancy over the set of random codes of size $N$ chosen in ${\cX_n}$ with uniform 
distribution. Let $X$ be a random vector such that $P(X=x)=2^{-n}$ for every $x\in {\cX_n}$ and let $\cZ=\{z_1,\dots,z_N\}$ be a subset 
formed of $N$ independent copies of $X.$ Equivalently, one can choose each $z_i$ by uniformly and independently selecting the values of
each of the $n$ coordinates. 

For a given $w, 1\le w\le n$ and $z_i\in \cZ$ the probability $\Pr(d(z_i,z_j)=w)=\binom nw 2^{-n},$ and
thus the expected number of pairs
  \begin{equation}\label{eq:E}
  {\sf E}|\{(z_i,z_j): d(z_i,z_j)=w\}|=N\,{\sf E}[A_w(\cZ)]=N(N-1)\binom nw 2^{-n}.
  \end{equation}
Now using \eqref{eq:closed}, we obtain that
   $$
   {\sf E}[\langle\lambda\rangle_{\cZ}]=\frac{N-1}{N}\frac{n}{2^{n+1}}\binom{2n}n.
   $$
Together with \eqref{eq:SH3} we conclude as follows.
\begin{proposition}\label{prop:random} The expected discrepancy of a random code of size $N$ in $\{0,1\}^n$ equals
   \begin{gather}
   {\sf E}[D^{L_2}(\cZ)]=\frac{n}{N2^{n+1}}\binom {2n} n=\sqrt{\frac{n}{\pi}}\frac{2^{n-1}}{N}(1-\alpha_n n^{-1}),\label{eq:ED}
   \end{gather}
where $\alpha_n<1$ is a constant.
\end{proposition}
\noindent To obtain the approximation in \eqref{eq:ED} we used standard inequalities for the central binomial coefficient (see \eqref{eq:center} below).

It is also easy to estimate the moments of $D^{L_2}(\cZ).$ For instance
   $$
   \text{Var}(A_w(\cZ))=N(N-1)\Big(\binom nw2^{-n}-\binom {n}{w} ^2 2^{-2n}\Big)\le {\sf E}A_w(\cZ).
   $$
Therefore, using \eqref{eq:closed}, \eqref{eq:E} and independence or pairwise distances, we obtain
  \begin{align*}
  \text{Var}(D^{L_2}(\cZ))=\frac{1}{N^2}\sum_{w=1}^n \lambda(w)\text{Var}(A_w(\cZ))\le \frac{N-1}{N}\frac{n}{2^{n+1}}\binom {2n}{n}
     - {\sf E}[D^{L_2}(\cZ)]
  \end{align*}

Similar results can be obtained if we limit ourselves to random linear subspaces of ${\cX_n}$ of a given dimension $k, 1\le k\le n.$ In this case, $N=2^k$ and
    $$
    {\sf E}A_w=\frac{2^k-1}{2^n-1}\binom nw,
    $$
which is essentially the same as \eqref{eq:E}.

Concluding, we have shown that the expected discrepancy is inverse proportional to the relative size of the random subset $\cZ$ in ${\cX_n}$ irrespective of whether $\cZ$ is a linear subspace or a fully random subset of ${\cX_n}.$

For the spherical case, expected discrepancy of a random configuration of size $N$ was computed in \cite{BilykLacey2017}, which showed that it is proportional to the quotient of the average distance on the sphere and $N$. In our case, ${\sf E}[\langle\lambda\rangle_{\cZ}]=\Lambda_n/N,$
the quotient of the average value of $\lambda(x,y)$ and $N$.

\subsubsection*{\sc Extending a code} Let $Z\subset {\cX_n}$ be a code. For every vector $z=(z_1,\dots,z_n)\in Z$ find $z_{n+1}=\oplus_{i=1}^n z_i$ and 
adjoin this coordinate to $z$, forming a vector $z'=(z|z_{n+1}).$ The set of vectors $\{z'| z\in Z\}$ forms an extended code $Z^{\text{ex}}$ of length $n+1$. The discrepancy $D^{L_2}(Z^{\text{\rm ex}})$ can be easily found from $D^{L_2}(Z).$
\begin{proposition}  Let $Z$ be a linear code of length $n=2p-1$ and size $N,$ then
    \begin{equation}\label{eq:extended}
    D^{L_2}(Z^{\text{\rm ex}})=2D^{L_2}(Z)+\frac1{2^{n+1}}\binom {2n}{n}.
    \end{equation}
  \end{proposition}
\begin{proof}
Let $A(Z)=\{1,A_1,\dots,A_n\}$ be the distance distribution of $Z$, then the distance distribution of the extended code is
$\{1,A'_i,i=1,\dots,2p\},$ where $A'_{2j-1}=0, A'_{2j}=A_{2j-1}+A_{2j}, j=1,\dots,p-1,$ and $A'_{2p}=A_{2p-1}.$ From \eqref{eq:SH3} and \eqref{eq:eo}
   $$
   D^{L_2}(Z)=\frac{n}{2^{n+1}}\binom {2n}{n}-\frac{1}{N}\sum_{i=1}^{p+1}(A_{2i-1}+A_{2i})\lambda(2i),
   $$
   where we formally put $A_{2p}=0.$
When the code is extended, the length increases by one, and the value of $\lambda(w)$ doubles. We obtain
   \begin{align}
   D^{L_2}(Z^{\text{ex}})&=\frac{n+1}{2^{n+2}}\binom{2n+2}{n+1}-\frac1N \sum_{j=1}^{p+1} 2A'_{2j}\lambda(2j)\nonumber\\
   &=\frac{2n+1}n\Lambda_n-\frac2N\sum_{i=1}^{p+1}(A_{2i-1}+A_{2i})\lambda(2i)\\
   &=2D^{L_2}(Z)+\frac 1n\Lambda_n.
   \end{align}
Upon substituting \eqref{eq:av},we obtain \eqref{eq:extended}.
\end{proof}

\subsection{Krawtchouk polynomials} 
Krawtchouk polynomials form a family of discrete orthogonal polynomials on $\{0,1,\dots,n\}$ with respect to the weight $\binom ni 2^{-n}.$
A Krawtchouk polynomial of degree $k$ is defined as 
   $$
   K_k^{(n)}(x)=\binom nk{}_{2}\hspace*{0pt}F_{1}(-k,-x;-n;2)
   $$
    \cite[p.183]{Ismail2005}, \cite[p.237]{KLS2010}. We note 
that our definition differs from the standard one by a factor $\binom nk,$ which gives $K_k^{(n)}(0)=\binom nk$ (the standard normalization gives
the value 1 at $x=0$). In this section we list properties of the Krawtchouk polynomials used below in our derivations.

The explicit expression for the polynomial of degree $k=0,1,\dots,n$ is as follows:
  \begin{equation}\label{eq:Kk}
  K_k^{(n)}(x)=\sum_{i=0}^k(-1)^i\binom xi \binom{n-x}{k-i}.
  \end{equation}
The orthogonality relations have the form
  \begin{gather}
  \langle K_i^{(n)},K_j^{(n)}\rangle:=\sum_{l=0}^n \binom nl K_i^{(n)}(l)K_j^{(n)}(l)=2^n\binom ni\delta_{ij} \label{eq:ort}
  \end{gather} 
and thus $\|K_k^{(n)}\|^2=\binom nk.$    

From \eqref{eq:Kk} it is easily seen that
  \begin{gather}\label{eq:n/2}
  K_k^{(n)}(x)=(-1)^kK_k^{(n)}(n-x)
  \end{gather}
The generating function of the numbers $K_k^{(n)}(x)$ for integer $x$ has the form
   \begin{equation}\label{eq:gf}
   \sum_{k=0}^n K_k^{(n)}(x)z^k=(1+z)^{n-x}(1-z)^x.
   \end{equation}
Simple rearranging of the binomial coefficients in \eqref{eq:Kk} yields the following symmetry relation:
  \begin{equation}\label{eq:symmetry}
  \binom ni K_k^{(n)}(i)=\binom nk K_i^{(n)}(k).
  \end{equation}
  Rewriting \eqref{eq:ort} with the help of \eqref{eq:symmetry}, we obtain the expansion of the function $\delta_i(w):=\bb(w=i), i=0,1,\dots,n$ into the basis $(K_k)$:
  \begin{equation}\label{eq:deltawi}
  \delta_i(w)=2^{-n}\sum_{k=0}^n K_i^{(n)}(k) K_k^{(n)}(w), \quad w=0,1,\dots,n.
  \end{equation}
The Krawtchouk polynomials satisfy the following Rodrigues-type formula
   \begin{equation}\label{eq:Rod}
   \binom nx K_k^{(n)}(x)=\binom nk \nabla^k \Big[\binom{n-k}{x}\Big],
   \end{equation}
where $\nabla f(x):=f(x)-f(x-1)$ is the finite difference operator, \cite{KLS2010}, Eq.(9.11.10).

The following relation was proved in \cite{FF1996}, Thm.~3.1.3:
  \begin{equation}\label{eq:SoS}
  \sum_{k=0}^n (K_k^{(n)}(i))^2=\binom{2n-2i}{n-i}\binom{2i}i/\binom ni.
  \end{equation}
A self-contained proof relying only on \eqref{eq:gf} is given in the Appendix. The proof in \cite{FF1996} is more involved and relies on different methods.
     
The next lemma is a particular case of a general result in the theory of spherical harmonics. We give a short proof for completeness.
\begin{lemma} Let $x,y\subset \cX_n$ be such that $d(x,y)=w.$ Then the convolution $K_k^{(n)}\ast K_m^{(n)}$ defines a radial 
kernel on $\cX_n$ according to the following identity:
     \begin{equation}\label{eq:convolution}
    \sum_{z\in\cX_n} K_k^{(n)}(d(x,z))K_m^{(n)}(d(z,y))=2^n K_k^{(n)}(w)\delta_{k,m}.
     \end{equation}
\end{lemma}
   \begin{proof} Let $u_1,u_2\in\cX_n.$ As is easily seen, $K_k(d(u_1,u_2))=\sum_{v\in\cX_n: |v|=k}(-1)^{(v,u_1+u_2)},$ where $(v,u)=\sum_{i=1}^n v_i u_i,$ the sum is evaluated mod 2, and where $|\!\cdot\!|$ denotes the Hamming weight. Thus,
   \begin{align*}
   \sum_{z\in \cX_n}K_k^{(n)}(d(x,z))K_m^{(n)}(d(z,y))&=\sum_{z\in\cX_n}\sum_{|u|=k}(-1)^{(u,x+z)}\sum_{|v|=m}(-1)^{(v,z+y)}\\
   &=\sum_{|u|=k}\sum_{|v|=m}(-1)^{(u,x)+(v,y)}\sum_{z\in\cX_n}(-1)^{(u+v,z)}\\
   &=2^n\sum_{|u|=k}(-1)^{(u,x+y)}\delta_{k,m}=2^n K_k^{(n)}(d(x,y))\delta_{k,m}. \qedhere
   \end{align*}
    \end{proof}

\subsection{Dual view of discrepancy} To a code $Z\subset {\cX_n}$ one associates a pair of distance distribution vectors, $A(Z)$ defined above, and a dual
distribution $A^\bot(Z)=(A_0^\bot,\dots,A_n^\bot).$ The vectors $A(Z)$ and $A^\bot(Z)$ are connected by the following {\em MacWilliams identities}:
    \begin{gather}\label{eq:MW1}
    A_w^\bot=\frac 1{N}\sum_{i=0}^n K_w^{(n)}(i) A_i, \quad w=0,1,\dots,n\\
    A_i=\frac{N}{2^n} \sum_{w=0}^n K_i^{(n)}(w)A_w^\bot,\quad i=0,1,\dots,n.\label{eq:MW2}
  \end{gather}
In the context of spherical sets, i.e., subsets of $S^{d}(\reals),$ the quantities analogous to $A^\bot_w$ are called {\em moments} of the code \cite{BDHDSS2019}, \cite[p.~200]{BHS2019}. In the particular case that $Z\in {\cX_n}$ is a linear $k$-dimensional subspace, the dual code of $Z$ 
is defined as $Z^\bot=\{x\in {\cX_n}\mid (x,z)=0 \text{ for all } z\in Z\}$ where $(\cdot,\cdot)$ is the inner product modulo 2. 
Then the dual distance distribution of $Z$ equals the distance distribution of $Z^\bot,$ i.e., $A^\bot(Z)=A(Z^\bot)$
 \cite[Eq.(5.13)]{mac91}. 
 
Expressing the discrepancy of the code $Z$ via the dual distance distribution may simplify the computation because it is often the case 
that either $Z$ or $Z^\bot$ has only a small number of distances. Substituting \eqref{eq:MW2} into \eqref{eq:SH3}, we obtain the relation
  \begin{align}
    \label{eq:DD}
    D^{L_2}(Z)&=\Lambda_n -\frac 1{2^n}\sum_{i=0}^n A_i^\bot\sum_{w=0}^n K_w^{(n)}(i)\lambda(w)\\
    &=-\frac 1{2^n}\sum_{i=1}^n A_i^\bot\sum_{w=0}^n K_w^{(n)}(i) \lambda(w), \label{eq:not0}
  \end{align}
where \eqref{eq:not0} follows from \eqref{eq:closed}, \eqref{eq:av}, and the fact that $A_0^\bot=1.$ Expression \eqref{eq:not0} may be preferable over
\eqref{eq:SH3}, \eqref{eq:DD} because these formulas involve subtraction of two large numbers, while \eqref{eq:not0} gives an explicit form of their
difference. Note that relations \eqref{eq:DD} and \eqref{eq:not0} are valid for linear as well as unrestricted subsets $Z$.

The sum on $w$ in \eqref{eq:not0} can be written in a different form.
Namely, a calculation involving the generating function \eqref{eq:gf} shows that the following identities 
   \begin{align}\label{eq:conj}
\sum_{w=0}^n K_w^{(n)}(i)2^{n-w} w\binom{w-1}{\lceil\frac w2\rceil-1}=(-1)^i\sum_{w=0}^{n-1} K_w^{(n-1)}(2i-2)\binom{n-1}w
   \end{align}
hold true for all $n\ge 1, 1\le i\le \frac12(n+1).$ 
   
\vspace*{.05in} {\em Remarks:}  1. Since $\lambda(2j-1)=\lambda(2j),$ we can combine the consecutive Krawtchouk numbers in \eqref{eq:not0} using a standard 
relation $K_{2j-1}^{(n)}(i)+K_{2j}^{(n)}(i)=K_{2j}^{(n+1)}(i),$ however, this does not seem to lead to further simplifications.

2. As a side observation, we note another possible interpretation of the numbers on the left (or on the right) in \eqref{eq:conj}. Denote them by $a_i(n),n=1,2,\dots.$  Apparently, the coefficients of the power series expansion
    $$
  (1-4x)^{(2i-3)/2}=1+\sum_{n\ge 1} c_nx^n 
  $$ 
are given by $c_n=(-1)^i a_i(n+1)$ for all $n\ge i.$ Observe that the sequence $(a_i(n), n\ge 1)$ for different values of $i$ is related to sequences \href{http://oeis.org/A002420}{A002420}-\href{http://oeis.org/A002424}{A002424} in OEIS \cite{oeis}. 
   
\vspace*{.05in} Relations \eqref{eq:not0}, \eqref{eq:conj} sometimes enable one to compute the discrepancy of the code $Z$ in closed form. For instance, 
let $n=2^m-1$ and suppose that $Z$ is the Hamming code. Its size is $N=2^{n-m}$ and $A_i^\bot=2^m-1$ if $i=2^{m-1}=(n+1)/2$ and
$A_i^\bot=0$ for all other positive $i$ (\cite{mac91}, \S1.9). We obtain the following statement. 
\begin{theorem}\label{thm:Ham} The quadratic discrepancy of the Hamming code $Z=\cH_m$ of length $n=2^m-1,m\ge 2$ equals
   \begin{equation}\label{eq:DH}
   D^{L_2}(\cH_m)=\frac n{2^{n}}\binom{n-1}{\frac{n-1}2}.
   \end{equation}
   For large $n$ the discrepancy $D^{L_2}(\cH_m)= \sqrt {n/4\pi}(1-o(1)).$
\end{theorem}
\begin{proof} The right-hand side of \eqref{eq:not0} contains a single nonzero term for $i=2^{m-1}=(n+1)/2.$ Substituting \eqref{eq:conj} in \eqref{eq:not0},
we obtain
   \begin{align*}
   D^{L_2}(\cH_m)&=-\frac n{2^n}\sum_{w=0}^{n}K_w^{(n)}((n+1)/2)\lambda(w)\\
   &=-\frac n{2^n}\sum_{w=0}^{n-1}K_w^{(n-1)}(n-1)\binom{n-1}{w}.
   \end{align*}
From \eqref{eq:Kk} we observe that $K_k^{(n)}(n)=(-1)^k \binom nk,$ so we obtain
  \begin{align*}
D^{L_2}(\cH_m)& =-\frac n{2^n}\sum_{w=0}^{n-1}(-1)^w\binom{n-1}{w}^2,
  \end{align*}
which turns into \eqref{eq:DH}   
upon engaging the identity $\sum_{i=0}^p (-1)^i\binom{p}{i}^2=(-1)^{p/2}\binom p{p/2}$ valid for even $p,$ and
noticing that $(-1)^{(n-1)/2}=-1.$
\end{proof}
Another proof of this theorem is given below after we develop a Fourier transform view of discrepancy.
Note that for the parameters of the Hamming code we find $|{\cX_n}|/N\approx n,$ and thus, $D^{L_2}(\cH_m)\approx \frac 1n{\sf E}[D^{L_2}(N)] $,
where ${\sf E}[D^{L_2}(N)]$ is the expected discrepancy given in \eqref{eq:ED}. 

For the dual code $\cH_m^\bot$ (the Hadamard, or simplex code) the discrepancy is found 
immediately from \eqref{eq:SH3} and the distance distribution given before the theorem. We obtain
  \begin{equation}\label{eq:simplex}
  D^{L_2}(\cH_m^\bot)=\Lambda_n-\frac{n}{N}\lambda((n+1)/2).
  \end{equation}

\vspace*{.1in}Let us give some numerical examples. 

\vspace*{.05in}
\begin{center}{{\small\sc Discrepancy of the Hamming codes and their duals\\[.05in]
\begin{tabular}{|l|ccccccc|}
\hline
\multicolumn{8}{|c|}{Hamming codes $\cH_m$, $n=2^m-1, N=2^{n-m}$}\\
\hline
$m$ &4&5&6&7&8&9&10\\
\hline
$D^{L_2}(Z)$ &1.571  &2.239& 3.179& 4.50471 &6.377& 9.027 &12.763\\
${\sf E}D^{L_2}(N)$ &17.336 &50.058 &143.016 &406.518 &1152.64 &3264.14 &9238.04\\
\hline
\multicolumn{8}{|c|}{Hadamard codes $\cH_m^\bot$, $n=2^m-1, N=2^{m}$}\\
\hline
$2^{-n}D^{L_2}(Z)$ &0.058&0.042&0.030&0.021&0.015&0.011&0.008\\
$2^{-n}{\sf E}D^{L_2}(N)$ &0.068&0.049 &0.035   &0.025  &0.018  &0.012  &0.009\\
\hline
\end{tabular}
}}\end{center}
\vspace*{.05in}

\noindent The Hamming codes form dense, regular packings of the space, and their discrepancy is much smaller than the average over all subsets of the 
same size. In Sec.~\ref{sec:LP} we show that they in fact minimize the discrepancy among all codes of the same cardinality. In contrast, the code $
\cH_m^\bot$ has only one nonzero distance, and its discrepancy approaches the average as $n$ increases 
(since the numbers are large, we scale them by $2^{-n}$).

To give one more example, the discrepancy of the Golay code of length $n=23, N=4096$ equals $390.75$ while ${\sf E}D^{L_2}(N)=2755.68,$
and again it is a minimizer of discrepancy among all codes of the same size.  

Many more examples can be generated since the distance distributions of many codes are known explicitly \cite{mac91} (and some of them
are conveniently listed online in OEIS \cite{oeis}). 

%
%
\subsection{Discrepancy and the sum of distances} \label{sec:sum}
The original form of the Stolarsky principle \eqref{eq:Ssph} connects $D^{L_2}(Z)$ for spherical sets with the sum of distances in $Z$. 
For the Hamming space, this is not exactly true, but is in fact true approximately. To begin, we note that the average distance in ${\cX_n}$ equals 
   $$
   \langle d\rangle_{{\cX_n}}=2^{-2n}\sum_{x,y\in{\cX_n}} d(x,y)=\frac n2.
   $$
Next, observe that the average distance in $Z$ equals 
   $$
   \langle d\rangle_Z=\frac1{N^2}\sum_{i,j=1}^Nd(z_i,z_j)=\frac 1{N}\sum_{w=1}^n w A_w.
   $$
The generating functions of the numbers $(A_w)$ and $(A_w^\bot)$ are related by the MacWilliams equation 
   $$
   \sum_{w=0}^n A_i y^i=\frac{N}{2^n}\sum_{w=0}^n A_w^\bot(1+y)^{n-i}(1-y)^i,
   $$
implied by \eqref{eq:MW2} and \eqref{eq:gf}.
Differentiating on $i$ and setting $y=1$ we obtain
   $$
    \langle d\rangle_Z=\frac n2-\frac{A_1^\bot}2.
    $$
Thus, $\langle d\rangle_Z\le \frac n2$ with equality if and only if $A_1^\bot=0.$

The central binomial coefficient can be bounded as follows:
  \begin{equation}\label{eq:center}
 \frac {c} {\sqrt{n\pi}}\le \frac{\binom {2n} n}{2^{2n}}\le\frac {1} {\sqrt{n\pi}},
   \end{equation}
where $c=c_n$ is a constant that is greater than $0.9$ for all $n\ge 2$ and tends to 1 as $n$ increases.
Substituting these estimates in \eqref{eq:SH3}, we obtain
\begin{proposition} For any code $Z\subset {\cX_n}$
   \begin{gather}
\frac{2^n}{\sqrt{\pi n}}\Big(c\frac n2 -\Big(\frac{n\langle d\rangle_Z}{2}\Big)^{\half}\Big)\le D^{L_2}(Z)\le \frac{2^n}{\sqrt{\pi n}}   \Big(\frac n2-\frac{c}{2}\langle d\rangle_Z\Big)\label{eq:C}\\
 D^{L_2}(Z)\le c'\frac{2^n}{\sqrt{\pi n}} \frac n2, \label{eq:C0}
  \end{gather}
where \eqref{eq:C0} holds if $A_1^\bot=0$, and $c'$ approaches $1/2$ as $n$ increases.
\end{proposition}
\begin{proof} From \eqref{eq:center} we obtain
  \begin{gather*}
  \frac{c2^n}{\sqrt{\pi n}}\frac n2\le\Lambda_n\le \frac{2^n}{\sqrt{\pi n}}\frac n2, \quad
  \frac{c 2^{n}}{\sqrt {\pi n}}i\le\lambda(2i)\le \frac{2^n}{\sqrt {\pi}}\sqrt{i}.
  \end{gather*}
For the upper bound in \eqref{eq:C} we compute
  $$
  D^{L_2}(Z)\le \frac{2^n}{\sqrt{\pi n}}\frac n2-\frac{c}{N\sqrt{\pi n}}\sum_{i=1}^{n/2}(A_{2i-1}2^{n-1}(2i-1)+A_{2i}2^{n-1}(2i))
  $$
(assuming $n$ is even), and this yields \eqref{eq:C}. The case of odd $n$ is similar. The lower bound is obtained from \eqref{eq:SH3} once we compute (again
assuming that $n$ is even)
  $$
  \frac1N\sum_{w=1}^nA_w\lambda(w)\le\frac{2^n}{\sqrt{\pi}}\sum_{i=1}^{n/2}\Big(\frac{A_{2i-1}}{N}+\frac{A_{2i}}{N}\Big)\sqrt i
  \le \frac{2^n}{\sqrt {2\pi}}\langle d\rangle_Z^{\half},
  $$
where the last step uses Jensen's inequality.
\end{proof}
The bounds in this proposition apply to any code of a given length, without accounting for the structure of the code.
The lower bound in \eqref{eq:C} trivializes if $\langle d\rangle_Z=\frac n2,$ but provides useful estimates in other cases. 

Let $d(N):=\min_{Z:|Z|=N}\langle d\rangle_Z$ be the smallest possible average distance over codes of a given size. The problem of 
bounding $d(N)$ was raised in \cite{Ahlswede1977} and was the subject of a number of follow-up papers. Under the assumption $N\le 2^{n-1}$ a bound
   $
   d(N)\ge \frac n2-\frac{2^{n-2}}{N}
   $
was proved in \cite{Fu2001}. Substituting it in \eqref{eq:C}, we can state the following result.
\begin{proposition} For any $Z\subset{\cX_n}, |Z|=N\le 2^{n-1}$
   $$
   D^{L_2}(Z)\le \frac{2^n}{\sqrt{\pi n}}\Big(c'\frac n2+\frac{c2^{n-3}}{N}\Big).
   $$
\end{proposition}
Other lower bounds on $d(N)$ are given in \cite{Fu2001} and subsequent works, with the best known results appearing in the recent paper \cite{YuTan2019}.

An example of configurations that minimize the average distance is given by subcubes in ${\cX_n}$ of codimensions 1 and 2. 
Let $Z=C_{n-m}:=\{0,1\}^{n-m}\times\{0\}^{m}$ be a subcube of ${\cX_n}.$ 
The distance distribution of $Z$ is $A_w=\binom{n-m}w, 1\le w\le n-m$ and $A_w=0, n-m<w\le n.$ The discrepancy of $Z$ equals
   $$
   D^{L_2}(C_{n-m})=\Lambda_n-\frac 1{2^{n-m}}\sum_{w=1}^{n-m} \binom{n-m}w\lambda(w).
   $$
For $m=1$ this can be evaluated in closed form, for instance by computer \cite{RISC}, and we obtain
    $$
    \sum_{w=1}^{n-1} \binom{n-1}w\lambda(w)=(n-1)\binom{2n-2}{n-1}
    $$
(Ap{\'e}ry numbers, \href{http://oeis.org/A005430}{A005430}) and
$$
D^{L_2}(C_{n-1})=\frac{n}{2^{n+1}}\binom {2n}{n}-\frac{n-1}{2^{n-1}}\binom{2n-2}{n-1}.
$$
The question whether subcubes are also discrepancy minimizers is likely resolved in the negative, see the discussion in Section~\ref{sec:LP} below.

\section{A Fourier transform view of discrepancy}\label{sec:Fourier}
In this section we derive a representation of the discrepancy $D^{L_2}(Z)$ in the transform domain. 
In view of \eqref{eq:kernel} this amounts 
to representing the kernel $\lambda(x,y)$ \eqref{eq:lambda} as a linear combination of the Krawtchouk polynomials. A direct approach is to
compute the inner product of the expression \eqref{eq:lambda} with $K_k^{(n)}$ for all $k,$ but this looks difficult. At the same time, from
\eqref{eq:cap2} it suffices to find the expansion of $\mu_t(x,y)=|B(x,t)\cap B(y,t)|$ (cf. \eqref{eq:mut}) and then ``integrate'' on $t$. 
Let $\phi_t=\mathbbm{1}_{\{0,1,\dots,t\}}$ be the indicator function of the set $\{0,1,\dots,t\},$ then $\mu_t=\phi_t\ast\phi_t$ is a convolution 
square; in more detail,
    \begin{equation}\label{eq:conv}
    \mu_t(x,y)=\sum_{z\in{\cX_n}}\phi_t(d(x,z))\phi_t(d(z,y)).
    \end{equation}
\begin{lemma}\label{thm:Bt} Let $x,y\in {\cX_n}$ be such that $d(x,y)=w.$
The Krawtchouk expansion of the kernel $\mu_t(x,y), t=0,\dots,n$ has the following form:
  \begin{equation}\label{eq:Bt}
  \mu_t(x,y)=2^{-n}\sum_{k=0}^n c_k(t)^2K_k^{(n)}(w),
  \end{equation}
where
  \begin{gather}
  c_0(t)=\sum_{i=0}^t \binom ni \label{eq:0}\\
  c_k(t)=\begin{cases} K_{t}^{(n-1)}(k-1), &t=1,\dots,n-1\\0& t=n\end{cases}, \quad k=1,\dots,n.  \label{eq:1n}
  \end{gather}
\end{lemma}
\begin{proof}
Let $\phi_t(l)$ be the function defined before the lemma, and let 
  $$
  \phi_t(l)=2^{-n}\sum_{k=0}^n c_k(t) K_k^{(n)}(l), \quad l=0,1,\dots,n
  $$
be its Krawtchouk expansion, where $c_k=\langle\phi_t,K_k^{(n)}\rangle/\binom nk.$
Since $K_0^{(n)}\equiv1,$ we obtain
    \begin{gather*}
     c_0(t)=\sum_{i=0}^t \binom ni .
         \end{gather*}
Further, for $k=1,\dots,n;\; t\le n-1$ we compute
      \begin{align}
    c_k(t)&=\frac 1{\binom nk}\sum_{i=0}^t\binom ni K_k^{(n)}(i) 
    =\frac 1{\binom nk}\frac nk \binom{n-1}tK_{k-1}^{(n-1)}(t), \label{eq:1n-1}\\
    &=\frac{1}{\binom{n-1}{k-1}}\binom{n-1}tK_{k-1}^{(n-1)}(t), \nonumber\\
    &=K_t^{(n-1)}(k-1),\nonumber
    \end{align}
where the expression for the sum in \eqref{eq:1n-1} follows by \eqref{eq:Rod} and the transition to the last line uses \eqref{eq:symmetry}.
Finally, for $t=n$ from \eqref{eq:n/2} we obtain
   $$
   c_k(n)=\frac 1{\binom nk}\sum_{i=0}^n (-1)^i\binom ni=0.
   $$
Now from \eqref{eq:conv} we obtain
  \begin{align*}
  \mu_t(x,y)&=\sum_{z\in\cX_n}\phi_t(d(x,z))\phi_t(d(z,y))\\
     &=2^{-2n}\sum_{z\in\cX_n} \sum_{k=0}^n c_k(t) K_k^{(n)}(d(x,z))\sum_{m=0}^n c_m(t) K_m^{(n)}(d(z,y))\\
     &=2^{-2n}\sum_{k=0}^n\sum_{m=0}^n c_k(t)c_m(t)\sum_{z\in\cX_n}K_k^{(n)}(d(x,z))K_m^{(n)}(d(z,y)),
  \end{align*}
which together with  \eqref{eq:convolution} yields \eqref{eq:Bt}.
\end{proof}

Calculations of the Fourier expansion of the intersection of metric balls form a recurrent topic in papers devoted to uniformly distributed sets;
see \cite[Eq.(10)]{Brauchart2019} for the spherical case, \cite[Eq.(4.52)]{Skriganov2019} for general two-point homogeneous spaces, and \cite[Appendix A]{Torquato2003} for $\reals^n.$
Additionally, in the spherical case, a function analogous to $\mu(x,y)$ was studied in \cite{Brauchart2013} in the context of Hilbert space reproducing kernels.
Casting our results in their language, we note that $\mu(x,y)$ is a reproducing kernel for the space of real functions $f$ on ${\cX_n}$ 
representable in the form
   \begin{equation}\label{eq:g1}
   f(x)=\sum_{t=0}^n \sum_{u\in {\cX_n}} g(u,t) \mathbbm{1}_{B(u,t)}(x)
   \end{equation}
with respect to the inner product $( f_1,f_2)=\sum_t\sum_u g_1(u,t) g_2(u,t),$ viz.,
   $$
   ( \mu(\cdot,y),f)=f(y).
   $$

Lemma \ref{thm:Bt} immediately implies a Krawtchouk expansion for the kernel $\lambda(x,y)$.
\begin{corollary}\label{cor:Klambda} Let $x,y\in{\cX_n}$ be such that $d(x,y)=w.$ We have
   \begin{gather}\label{eq:Klambda}
     \lambda(x,y)=\lambda(w)=\sum_{k=0}^n\hat\lambda_k K_k^{(n)}(w)\\
     \hat\lambda_0=\Lambda_n, \;\hat\lambda_k=-2^{-n}\binom{2n-2k}{n-k}\binom{2k-2}{k-1}/\binom {n-1}{k-1}, \;k=1,2,\dots,n,\label{eq:cf}
    \end{gather}
and thus the kernel $(-\lambda(x,y))$ is positive definite up to an additive constant.
\end{corollary}
\begin{proof}
Starting with \eqref{eq:cap2} and using \eqref{eq:Bt}, we find that
 \begin{align*}
 \lambda(x,y)&=2^{n-1}(n+2)-\sum_{t=0}^n\mu_t(x,y)\\
 &=2^{n-1}(n+2)-2^{-n}\sum_{t=0}^n\Big(\sum_{i=0}^t\binom ni\Big)^2-2^{-n}\sum_{k=1}^{n}\sum_{t=0}^{n-1}c_k(t)^2 K_k^{(n)}(w).
 \end{align*}
On account of \eqref{eq:sqsum}, \eqref{eq:closed}, \eqref{eq:av}, the constant term $\hat\lambda_0=\Lambda_n$, and 
  $$
  \hat\lambda_k=-2^{-n}\sum_{t=0}^n c_k(t)^2, k=1,\dots,n.
  $$
Since $c_k(n)=0,$ we obtain $\hat\lambda_k=-2^{-n}\sum_{t=0}^{n-1} K_t^{(n-1)}(k-1)^2,$ and by \eqref{eq:SoS} this yields \eqref{eq:cf}.

\end{proof}

\vspace*{.2in}\hspace*{-.2in}
\includegraphics[width=.5\linewidth]{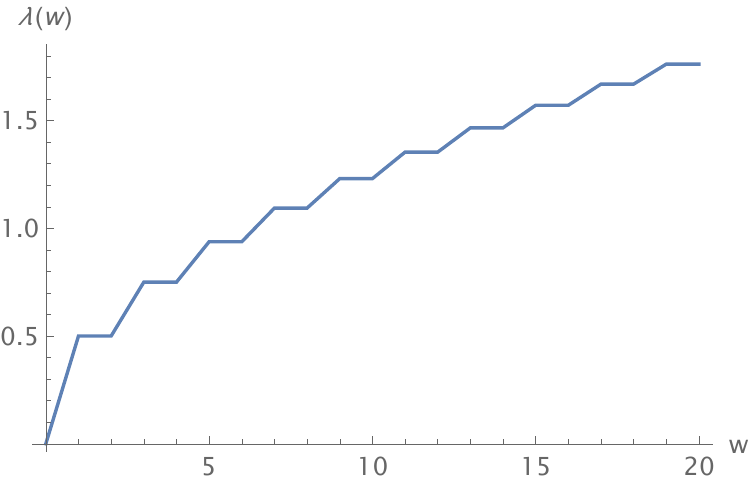}\hspace*{.3in}
\includegraphics[width=.45\linewidth]{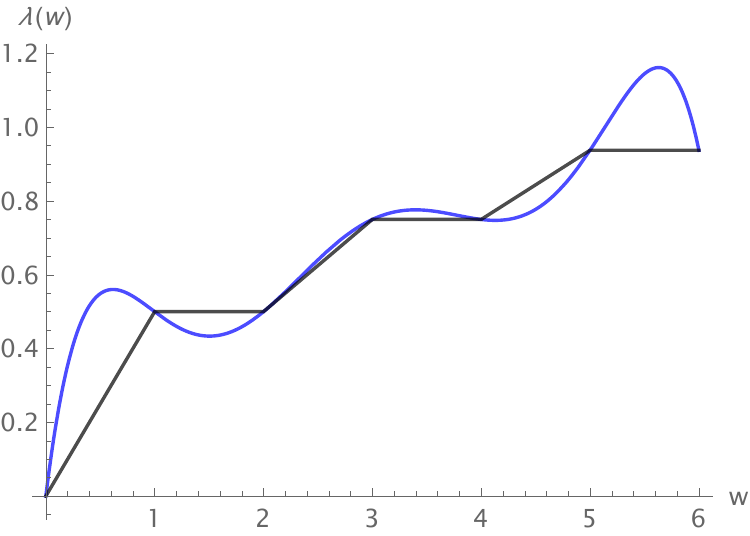}

\vspace*{.1in}\noindent\hangindent.2in\hangafter=1{\footnotesize{\sc Fig.1:} The plots show $\lambda(w)$ for $n=20$ (left figure) and $n=6$ (right figure). In the right plot we also show the polynomial \eqref{eq:Klambda} that is equal to $\lambda(w)$ at integer values of $w$. The plots are scaled by $2^{-n}$.\par}
\vspace*{.2in}

{\em Remarks:} 1. The constant coefficient of the Fourier expansion is the expectation of the function with respect to the
underlying measure, and this is indeed the case in \eqref{eq:Klambda}; cf.~\eqref{eq:av}.

2. The fact that $\lambda$ is an invariant negative definite kernel implies, independently of \eqref{eq:kernel}, that $\langle\lambda\rangle_{\cX_n}\ge 
\langle\lambda\rangle_Z$ for any $Z\subset{\cX_n}$. This is the well-known ``inequality about the mean'' \cite{kab78}.

3. Taking $w=0$ in Eq.~\eqref{eq:Klambda} we can rewrite the expansion for $\lambda$ in the following form:
    $$
    \lambda(w)=\sum_{k=1}^n\hat \lambda_k(K_k^{(n)}(w)-K_k^{(n)}(0)).
    $$

It is easy to check that the coefficients $\hat\lambda_k$ are symmetric with respect to the middle, and their absolute values 
decrease for $k<n/2$ and increase for $k>n/2$.
\begin{lemma}\label{lemma:monmid} We have
    \begin{gather*}
    \hat\lambda_i=\hat\lambda_{n-i+1}, \quad i=1,\dots,n/2;\; n\text{ even}\\
     \hat\lambda_{\frac{n+1}2-i}=\hat\lambda_{\frac{n+1}2+i}; \quad i=1,\dots,\frac{n-1}2;\; n\text{ odd}\\
     \max_{1\le k\le n}\hat\lambda_k=\begin{cases}
     \hat\lambda_{\frac n2}=\hat\lambda_{\frac n2+1}=-2^{-n}\frac{n}{2n-2}\binom n{n/2}, &n\text{ even}\\
     \hat\lambda_{\frac{n+1}2}=-2^{-n}\binom{n-1}{\frac{n-1}2},&n\text{ odd}.
     \end{cases}
     \end{gather*}
  
\end{lemma}
\begin{proof} The statements about the symmetry are checked directly using \eqref{eq:cf}. To prove the claim about $\max_k\hat\lambda_k$, we compute
   $$
   \frac{\hat\lambda_k}{\hat\lambda_{k-1}}=\frac{2k-3}{2n-2k-1}.
   $$
For $n$ odd this quantity is less than one for $k\le (n+1)/2$ and greater than one for $k\ge (n+3)/2.$ Thus the sequence $(|\hat\lambda_k|)_k$ is monotone decreasing till $k=(n+1)/2$ and monotone increasing after that, with a unique minimum at $(n+1)/2.$ Similarly, for $n$ we observe that $\hat \lambda_{n/2}=\hat\lambda_{\frac n2+1}$ and both attain the maximum value. Finally, the expressions for this value follow directly from \eqref{eq:cf} for both odd and even $n$.
\end{proof}

\vspace*{.05in}We are now in a position to compute a transform-domain representation of the discrepancy of the set $Z\subset{\cX_n}.$ The following theorem, which gives an explicit form of the formula for discrepancy found earlier in \eqref{eq:not0}, is immediate.
  \begin{theorem} Let $Z\subset {\cX_n}$ be a code of size $N$ with dual distance distribution $(A_k^\bot, k=0,\dots,n)$ \eqref{eq:MW2}. Then
    \begin{align}
        D^{L_2}(Z)&=2^{-n}\sum_{k=1}^n\frac{\binom{2n-2k}{n-k}\binom{2k-2}{k-1}}{\binom {n-1}{k-1}} A_k^\bot \label{eq:Ddual}
      \end{align}
\end{theorem}

For instance, for the Hamming code the dual distance distribution contains only one nonzero coefficient, $A_{(n+1)/2}^\bot=n.$
From \eqref{eq:Ddual} we obtain
     \begin{equation}\label{eq:DHF}
     D^{L_2}(\cH_m)=-n\hat\lambda_{\frac{n+1}2}.
     \end{equation}
which matches the expression for $D^{L_2}(\cH_m)$ derived earlier in \eqref{eq:DH}. 

As noted in Remark 1 before Lemma~\ref{lemma:monmid}, Fourier expansion of the quadratic discrepancy \eqref{eq:Ddual} does not include the constant term; see also \cite{Skriganov2019} for the case of the Euclidean sphere.

\section{Discrepancy as a potential function on \texorpdfstring{${\cX_n}$}{}}\label{sec:LP}

In this section we study configurations in $\{0,1\}^n$ that minimize discrepancy among all codes of the same cardinality.
Trivial examples are given by $Z=\{0,1\}^n$ and $Z=\{x\},$ where $x$ is any point in the space, which have discrepancy 0 (for the entire
space this follows from the definition and for singletons this is implied by \eqref{eq:not0}).

Let $Z\subset{\cX_n}$ be a code and let $f(x,y)=f(d(x,y))$ be a function on ${\cX_n}\times{\cX_n}$ that depends on the distance between the arguments. 
The {\em potential energy} of $Z$ with respect to
$f$ is defined as
    \begin{equation}\label{eq:energy}
    E_f(Z)=\frac1N\sum_{\begin{substack}{i,j=1\\i\ne j}\end{substack}}^N f(d(z_i,z_j)).
    \end{equation}
To relate discrepancy of the code to this definition, define a potential function 
  $f:\{1,\dots,n\}\to \reals$ by setting 
 \begin{equation}\label{eq:f}
     f(x,y)=f(d(x,y))=\Lambda_n-\lambda(d(x,y)).
 \end{equation}
 By Corollary \ref{cor:Klambda}, the function $f$ is positive definite, i.e., it is contained in the nonnegative cone of the Krawtchouk basis.
The discrepancy of the code $Z$ equals
   $$
   D^{L_2}(Z)=\frac 1N(\Lambda_n+E_f(Z)).
   $$
Thus, $D^{L_2}(Z)$ is proportional to the potential energy of the code with respect to $f$ (up to an additive term which can be removed by adjusting the definition of $f$).

Universal bounds on the discrepancy of a code of size $N$ can be obtained using linear programming. For the ease of writing, we 
consider an LP problem for the maximum value of $\langle \lambda\rangle_Z$ rather than the minimum of $E_f$.
By Stolarsky's invariance, Theorem \ref{thm:SHF}, the value of the linear program
  \begin{gather}
 \max \Big\{\sum_{k=1}^n A_k\lambda(k)\;\Big|\;
  \sum_{k=1}^n A_kK_i^{(n)}(k)\ge -\binom ni, i=1,\dots,n;
  \sum_{k=1}^n A_k=N-1, A_k\ge 0\Big\} \label{eq:primal}
  \end{gather}
gives a lower bound to $D^{L_2}(Z)$, and so does any feasible solution of the dual LP problem. The dual problem has the following form:
  \begin{equation}\label{eq:dualLP}
  \min\Big\{\sum_{i=0}^n \binom ni h_i-h_0 N\;\Big|\; \sum_{i=0}^n h_iK_i^{(n)}(k)\le-\lambda(k), k=1,\dots,n; h_i\ge0, i=1,\dots,n\Big\}
  \end{equation}
where we used the fact that $K_k^{(n)}(0)=\binom nk.$ This approach is rooted in the works of Delsarte, and the
first paper to apply it to the problem of estimating energy of point sets was \cite{Yudin1993} which addressed the case of spherical codes. For finite spaces an analogous bound on the energy of the code was derived in \cite{ash01b} which showed the following
proposition (we state it for the case of $f=\lambda$).

\begin{proposition}\label{prop:LP} Let $Z\subset\{0,1\}^n$ be a code of size $N$ and let $E_\lambda(Z):=\frac{1}{N}\sum_{i,j}\lambda(d(z_i,z_j))$ 
be the energy associated with the potential function $\lambda(i).$ 
Let $h(i)=\sum_{k=0}^n h_kK_k^{(n)}(i)$ be a polynomial on $\{0,1,\dots,n\}$ such that (a), $h_k\ge 0$ for all $k\ge 1$ such 
that $A_k^\bot >0$ and
$(b)$, $h(i)\le -\lambda(i)$ for all $i\ge 1$ such that $A_i>0.$ Then
   \begin{equation}\label{eq:energybound}
   E_\lambda(Z)\le h(0)-Nh_0
   \end{equation}
with equality if and only if all the inequalities in the assumptions (a),(b) are satisfied with equality.
\end{proposition}
Clearly, the bound \eqref{eq:energybound} is a rephrasing of the dual problem \eqref{eq:dualLP}, and the 
conditions for equality are just the corresponding complementary slackness conditions. Let $D^{L_2}(n,N)$ be the minimum discrepancy 
\eqref{eq:Dmin}, then
  \begin{equation}\label{eq:E-D}
  D^{L_2}(n,N)=\Lambda_n-\frac 1N E_\lambda(n,N),
  \end{equation}
   where
   $$
   E_\lambda(n,N)=\max_{Z\subset\{0,1\}^n, |Z|=N
   } E_\lambda(Z).
   $$
A polynomial that satisfies the constraints of the problem \eqref{eq:dualLP}, gives a universal bound $D^{L_2}(Z)\ge D^{L_2}(n,N)$ for all codes of cardinality $N$ irrespective of their distance distribution (this is because these constraints are more stringent than in Proposition \ref{prop:LP}).

Finding feasible vectors for the problem \eqref{eq:dualLP} in some cases is aided by our knowledge of the Krawtchouk coefficients of the energy function given in \eqref{eq:cf}, see for instance the proof of \eqref{eq:LPHamming} below. This information can be used for constructing $h(x)$ as long as we satisfy the inequalities $h(i)\le f(i),$ although controlling these conditions is generally not immediate. We also note that, from \eqref{eq:monotone}, the function $\lambda(2i)$ is ``concave,'' i.e., $\Delta^2(\lambda(2i))<0$ for all $i\ge 1,$ where $\Delta^2$ is the second finite difference.

In the next theorem we give some simple bounds on $E_\lambda(N)$ which will be used in examples below.
  \begin{theorem}\label{thm:LPbounds} For any $N\ge 1$
   \begin{gather}
   E_\lambda(n,N)\le (N-1)\lambda(n).\label{eq:N}
     \end{gather}
For $n=2t-1, t\ge 2$ 
   \begin{equation}\label{eq:LPsimplex}
   E_\lambda(n,N)\le \begin{cases}\lambda(t)(N-\half)& t\text{ even}\\[.1in]
   \frac{\lambda(t)}{n+1}(Nn-(n-1)/2)& t\text{ odd.}
   \end{cases}
   \end{equation} 
For any $N\ge1$
    \begin{align}\label{eq:LPHamming}
    E_\lambda(n,N)\le 
       \begin{cases}
       N\Lambda_n+(2^n-N)\hat\lambda_{\frac {n+1}2} &n\text{ odd}\\
       N\Lambda_n+(2^n-N)\hat\lambda_{\frac n2} &n\text{ even}. 
       \end{cases}
    \end{align}   
   \end{theorem}
\begin{proof} 
We have argued earlier, after Theorem~\ref{thm:SHF},  that \eqref{eq:N} holds for $N=2.$ 
To prove it for any $N$, take the constant polynomial $h(x)=-\lambda(n).$
The monotonicity condition \eqref{eq:monotone} implies that $h(i)\le -\lambda(i)$ for all $i=1,\dots,n,$ and the claim follows from \eqref{eq:energybound}.

To show \eqref{eq:LPsimplex}, choose
  $
  h(x)=h_0+h_1(K_1^{(n)}(x)+K_n^{(n)}(x)),
  $
  where $h_1=\frac{\lambda(t)}{4t}$ and
  \begin{gather*}
  h_0=-\lambda(t), \; t \text{ even};\quad
  h_0=-\lambda(t)\Big(1-\frac1{2t}\Big),\; t \text{ odd}.
  \end{gather*}
To argue that $h(i)\le -\lambda(i)$ for all $i=1,\dots,n$ we note that $h(i)=\lambda(i)$ for $i$ about $n/2$ (for $t-1\le i\le t+2$ if
$t$ is even and $t-2\le i \le t+1$ if it is odd). The other conditions are confirmed by using the ``convexity'' of $-\lambda(i)$, relation \eqref{eq:monotone} and $K_1^{(n)}(x)=n-2x, K_n^{(n)}(l)=(-1)^l,l=0,1,\dots,n;$  we omit the details.
Finally, \eqref{eq:LPsimplex} is obtained by substituting $h_0$ and $h(0)=h_0+h_1(n+1)$ into \eqref{eq:energybound} and simplifying.

Now let us prove \eqref{eq:LPHamming}. For this, we take a polynomial of the following form:
   $$
  h(w)=-\lambda(w)+ m_n 2^n\delta_0(w)=-\sum_{k=0}^n \hat\lambda_k K_k(w)+m_n\sum_{k=0}^n K_k(w), \quad w=0,\dots,n, 
    $$
where $m_n:=\max_{1\le k\le n}\hat\lambda_k$ and the expansion of $\delta_0(w)$ is given in \eqref{eq:deltawi}.
By construction we have that the coefficients $h_k=m_n-\hat\lambda_k\ge 0,k=1,\dots,n$ and $h(w)=-\lambda(w), w=1,\dots,n.$ 
Further,
  $$
  h(0)=2^n m_n, \quad h_0=m_n-\hat\lambda_0=m_n-\Lambda_n,
  $$
and thus
  $$
  E_\lambda(n,N)\le h(0)-N h_0=2^n m_n-N(m_n-\Lambda_n)=N\Lambda_n+(2^n-N)m_n.  
  $$
Now \eqref{eq:LPHamming} follows upon substituting $m_n$ from Lemma \ref{lemma:monmid}.
\end{proof}

Using \eqref{eq:E-D}, we obtain the following bounds on discrepancy.

\begin{corollary}\label{cor:bounds} For any $N\ge 1$
    \begin{equation}\label{eq:B2}
    D^{L_2}(n,N)\ge \Lambda_n-\frac{N-1}N\lambda(n).
    \end{equation}
For $n=2t-1, N\ge 1$
   \begin{equation}\label{eq:Bsimplex}
   D^{L_2}(n,N)\ge \begin{cases}
     \Lambda_n-\frac{2N-1}{2N}\lambda(t) &t\text{ even}\\[.1in]
     \Lambda_n-\frac{Nn-(n-1)/2}{N(n+1)}\lambda(t) &t\text{ odd}.
     \end{cases}
   \end{equation}
For any $N\ge 1$
 \begin{equation}\label{eq:DHm}
  D^{L_2}(n,N)\ge \begin{cases}
   -(\frac{2^n}{N}-1)\hat\lambda_{\frac{n}2}, &n\text{ even}\\[.07in]
  -(\frac{2^n}{N}-1)\hat\lambda_{\frac{n+1}2}, &n\text{ odd}.
  \end{cases}
  \end{equation}
\end{corollary}
We point out that bounds \eqref{eq:N} and \eqref{eq:LPHamming} can be also obtained directly from \eqref{eq:SH3} and \eqref{eq:Ddual}, respectively, using the fact that $\lambda(w)$ is monotone nondecreasing as a function of $w$ and that the middle coefficient $\hat\lambda_{\frac{n}2}$
(or $\hat\lambda_{\frac{n+1}2}$) is the largest among the Krawtchouk coefficients of $\lambda(w).$ The bound \eqref{eq:Bsimplex} does not seem to have an immediate direct proof.

We continue with some examples of discrepancy-minimizing configurations. The examples also show that neither of the bounds \eqref{eq:Bsimplex} and \eqref{eq:DHm} is uniformly better than the other one.

1. The repetition code $Z=(0^n,1^n).$ From \eqref{eq:SH3} we have $D^{L_2}(Z)=\Lambda_n-\frac 12 \lambda(n),$ and $A_n=1$ and $A_k=0$ for $k=1,\dots,n-1.$ This matches the bound \eqref{eq:B2} with equality. Clearly the only case when the bound \eqref{eq:B2} can be attained is $N=2.$
   
2a. The Hamming code $Z=\cH_m, m\ge2.$ We have $n=2^m-1,N=2^{n-m}.$ Substituting these parameters into \eqref{eq:DHm}, we find
   $$
   D^{L_2}(n,N)=\Lambda_n-\frac 1NE_\lambda(N)\ge -n\hat\lambda_{\frac{n+1}2}
  $$
which exactly matches the value $D^{L_2}(\cH_m)$ given in \eqref{eq:DH}, \eqref{eq:DHF}. This confirms optimality of the Hamming code.

2b. The shortened Hamming code $Z=\cH_m^{\text s}, m\ge 2$. Shortening means taking a half of the codewords in $\cH_m$ that contain zero in some
fixed coordinate and removing this coordinate. This results in a code of length $n=2^{m}-2$ and cardinality $N=2^{n-m}.$ The dual code $Z^\bot$ has 
distance distribution $A_{n/2}^\bot=\frac{n+2}2, A_{(n/2)+1}^\bot=\frac n2,$ and $A^\bot_k=0$ for all other $k\ge 1.$ Substituting into \eqref{eq:Ddual} and using Lemma~\ref{lemma:monmid}, we conclude that $D^{L_2}(Z)$ meets the case of $n$ even in \eqref{eq:LPHamming}. This shows that the shortened Hamming code is a discrepancy minimizer.

Experimentally, also the twice shortened Hamming code is LP-optimal (is a discrepancy minimizer), although it does not meet the bounds in Corollary \ref{cor:bounds}. 

3. The following codes were found to be discrepancy minimizers by computer: 
\begin{enumerate}[label=(\roman*)]
   \item the Golay code with $n=23, N=4096,$
   \item the shortened Golay code,
   \item the twice shortened Golay code,
   \item the quadratic residue code with $n=17, N=512$ (\href{http://oeis.org/A028381}{A028381}) ,
   \item the 2-error-correcting BCH codes with $n=31,N=2^{21}$ and $n=127,N=2^{113}$ and their shortened codes.
\end{enumerate}
 It is likely that the results in (v) extend to all BCH codes of length $n=2^{m}-1, N=2^{n-2m}$ for $m$ odd. In this case, the dual distance distribution has three nonzero coefficients $A_w^\bot$ for $w=2^{m-1}\pm2^{(m-1)/2},w=2^{m-1}$ \cite{mac91}, and one needs to design
 a polynomial that equals $\lambda(w)$ for these values of $w$.
 
As noted in \cite{Barg2020a}, if a code $Z$ is a minimizer of discrepancy, then so is its complement $Z^c:=\cX_n\backslash Z.$

\vspace*{.1in} 
Recall that the repetition codes for odd $n$, codes formed of a single vector,  the Hamming codes,  and the binary Golay code of length 23
exhaust the list of all perfect codes in $\{0,1\}^n$ \cite[Sec.~6.10]{mac91}. This enables us to make the following observation.

\begin{theorem} Perfect codes in $\{0,1\}^n$ are discrepancy-minimizing configurations.
\end{theorem}

If the code is a discrepancy minimizer, its dual is not necessarily a minimizer or even LP-optimal for the discrepancy problem. Indeed, while the 
Hamming codes are optimal, this is not true for their duals, i.e., the Hadamard codes. For instance, the dual code $\cH_3^\bot$ of length $n=7$ has distance
distribution $A_4=7$ and $A_k=0$ for all other $k\ge 1.$ For the class of Hadamard codes we have $n=2^m-1, N=2^m, m\ge 2.$ 
Using the case of even $t$ in \eqref{eq:LPsimplex}, we obtain the bound
   $$
 E_\lambda(n,N)\le \lambda\Big(\frac{n+1}2\Big)(n+1/2)
   $$

\noindent and the code $Z$ spanned over ${\mathbb F}_2$ by $1110000,0011100,0000111$ meets it with equality. It has a strictly smaller
value of discrepancy than $\cH_3^\bot$ (123/32 vs. 315/32), and its  distance distribution vector $(0,0,3,2,1,1,0)$ is optimal for the LP problem. Similar examples can be constructed for larger $m.$

Some of the usual suspects such as the extended Golay code of length $n=24,$ the Nordstrom-Robinson code of length $n=16$ and other codes
in the family of Kerdock codes, Reed-Muller codes of small length, or the subcubes $C_{n-m}$ are not LP optimal. We did not attempt to examine the question whether the distance distributions produced by the linear program \eqref{eq:primal} for the parameters of these examples correspond to actual codes (a priori this is not guaranteed, and some authors resort to the term {\em quasicode} to account for this).

We conclude this section by collecting the bounds on $D^{L_2}(n,N)$ obtained in this paper. The following theorem proves
the inequalities mentioned in the Introduction.
\begin{theorem}\label{prop:bounds}
For large $n$ and $N=o(2^n)$ we have the asymptotic bounds
  \begin{equation}\label{eq:asymp}
    c\frac1{\sqrt{n}}\frac {2^n}{N}\le D^{L_2}(n,N)\le C\sqrt{n}\frac{2^n}{N}
 \end{equation}
for some constants $c,C.$ The discrepancy $D^{L_2}(n,N)$ is bounded away from zero unless $\frac{2^n}{\sqrt n}=o(N).$ If $N=2^{rn}, 0<r<1,$
then 
  $$
  (\log N)^{-1/2}N^\alpha\lesssim D^{L_2}(n,N)\lesssim (\log N)^{1/2}N^\alpha,
  $$
where $\alpha=\frac 1r-1.$  
\end{theorem}
\begin{proof}
The upper bound follows from Proposition \ref{prop:random} upon using \eqref{eq:center}. In the same way, the lower bound in \eqref{eq:asymp}
follows from \eqref{eq:DHm} and Lemma \ref{lemma:monmid}.
\end{proof}

\section{Extensions and open questions}\label{sec:open}
\subsection{Generalized Stolarsky's identities} We begin with a simple generalization of Theorem~\ref{thm:SG}, obtained when the definition of $D^{L_2}(Z)$ in \eqref{eq:D} is extended to a weighted sum. Define
   \begin{equation*}
   D^{L_2}_G(Z)=\sum_{t=0}^n g_t (D_{t}(Z))^2,
   \end{equation*}
where $G=(g_0,g_1,\dots,g_n)$ is a real vector and $D_t(Z)$ is given in \eqref{eq:DG}. The corresponding weighted version of the
invariance principle is given in the next theorem.
%
%
  \begin{theorem} Let $\cX$ be a finite metric space, let $Z\subset {\cX}$ be a subset of size $N,$ and let $g_i\ge 0, i=0,1,\dots,n$ and $\gamma(t):=\sum_{i=t}^n g_i.$ Then
   \begin{equation}\label{eq:GSP}
   D_G^{L_2}(Z)=\langle \lambda_G\rangle_{\cX}-\langle \lambda_G\rangle_{Z},
   \end{equation}
where for $x,y\in {\cX}$
  $$
  \lambda_G(x,y):=\frac12 \sum_{z\in {\cX}}|\gamma(d(x,z))-\gamma(d(y,z))|.
  $$   
  \end{theorem}
\begin{proof} The proof is close to the proof of Theorem \ref{thm:SG}. Similarly to \eqref{eq:cap1} we obtain
  \begin{equation}\label{eq:DGt}
  \sum_{t=0}^n g_t D_t(Z)^2=\frac 1{N^2}\sum_{i,j=1}^N \sum_{t=0}^n g_t|B(z_i,t)\cap B(z_j,t)| - \frac 1{|{\cX}|}\sum_{t=0}^n g_t |B(u,t)|^2 .
  \end{equation}
Since $g_i\ge 0$ for all $i$, the function $\gamma(t)$ is monotone
nonincreasing, and the calculation in \eqref{eq:cap2} takes the following form:
      \begin{align}
      \sum_{t=0}^n g_t|B(x,t)\cap B(y,t)|&=\sum_{t=0}^n g_t \sum_{z\in {\cX}} \mathbbm{1}_{B(x,t)}(z)\mathbbm{1}_{B(y,t)}(z)=
        \sum_{z\in{\cX}}\sum_{t=0}^n g_t \mathbbm{1}_{B(x,t)}(z)\mathbbm{1}_{B(y,t)}(z)\nonumber\\
        &=\sum_{z\in{\cX}}\sum_{t=\max(d(z,x),d(z,y))}^n g_t=\sum_{z\in {\cX}}\gamma(\max(d(z,x))(d(z,y)))\nonumber\\
        &=\sum_{z\in {\cX}}\min(\gamma(d(z,x)),\gamma(d(z,y)))\nonumber\\
        &=\sum_{z\in {\cX}}\frac12\{\gamma(d(z,x))+\gamma(d(z,y))-|\gamma(d(z,x))-\gamma(d(z,y))|\},\label{eq:Gcap2}
  \end{align}
where $u$ is any fixed point in $\cX.$ Similarly, Eq.~\eqref{eq:sq1} takes the form
  \begin{equation}\label{eq:Gsq1}
   \sum_{t=0}^ng_t|B(u,t)|^2=\frac1{|{\cX}|}\sum_{x,y\in{\cX}}\sum_{z\in \cX}
    \frac 12\{\gamma(d(x,z))+\gamma(d(y,z))-|\gamma(d(x,z))-\gamma(d(y,z))|\}.
  \end{equation}
Using \eqref{eq:Gsq1} and \eqref{eq:Gcap2} in \eqref{eq:DGt} finishes the proof.  
\end{proof}
 For connected spaces such as $S^d$, weighed versions of the invariance principle 
were earlier considered in \cite{Brauchart2013,Skriganov2020,Skriganov2019}.
Relation \eqref{eq:GSP} applies to any finite metric space. In the case of the Hamming space ${\cX_n}=\{0,1\}^n,$ we can write \eqref{eq:GSP} in a more
specific form relying on the results of the previous sections.  Rewriting \eqref{eq:Gcap2} we obtain
  \begin{equation}\label{eq:lambdaG}
  \lambda_G(x,y)=\sum_{z\in \cX}\gamma(d(z,u))-\sum_{t=0}^n g_t\mu_t(x,y),
  \end{equation}
where $\mu_t$ is given in \eqref{eq:mut}. One implication of this equality is a Krawtchouk expansion of the kernel $\lambda_G$.
\begin{proposition} The Krawtchouk expansion of the kernel $\lambda_G(x,y)$ has the form
  $$
  \lambda_G(w)=\sum\nolimits_{k=0}^n \hat\lambda_{G,k} K_k^{(n)}(w),
  $$
where $w=d(x,y),\;\hat\lambda_{G,0}=\langle\lambda_G\rangle_{\cX}$ and 
  $$
  \hat\lambda_{G,k}=-2^{-n}\sum\nolimits_{t=0}^{n-1}(K_t^{(n-1)}(k-1))^2 g_t, \quad k=1,\dots,n.
  $$
\end{proposition}
\begin{proof} Eq. \eqref{eq:lambdaG} implies that $\lambda_G(x,y)$ is a radial kernel on $\cX_n,$ and thus $\hat\lambda_{G,0}$ is the
average value of $\lambda_G$ on the space $\cX$ by definition. The coefficients $\hat\lambda_{G,k}, k\ge 1$ follow by a direct substitution from
Lemma~\ref{thm:Bt}.
\end{proof}

The weighted version of Stolarsky's invariance principle for the Hamming space can be written in the following form.
\begin{theorem}[{\sc Weighted Stolarsky's invariance}] Let $Z\subset \cX_n=\{0,1\}^n$ be a code with distance distribution $A(Z)$ and dual distance distribution $A^\bot(Z).$ Then
  \begin{align*}
      D_G^{L_2}(Z)&=\langle \lambda_G\rangle_{\cX}-\frac 1N\sum_{w=1}^n A_w\lambda_G(w)\\
      &=\frac1{2^n}\sum_{k=1}^n A_k^\bot \sum_{t=0}^{n-1}g_t(K_t^{(n-1)}(k-1))^2.
   \end{align*}
\end{theorem}

A further generalization, discussed in \cite{Bilyk2018,Brauchart2013}, 
suggests to replace the indicator function of the metric ball in \eqref{eq:DG} with an arbitrary radial function $f(d(x,y)):\{0,\dots,n\}\to \reals.$  In other words, rather than starting with the deviation from
the uniform distribution for spheres in ${\cX},$ we build the notion of discrepancy starting with a function on ${\cX}$. 
For a subset $Z\in {\cX}$ define 
   $$
   D^{L_2}_f(Z)=\sum_{x\in{\cX}}\Big(\frac1N\sum_{j=1}^N f(d(x,z_j))-\frac 1{|{\cX}|}\sum_{u\in {\cX}} f(d(x,u))\Big)^2.
   $$
Upon squaring on the right-hand side, this function gives rise to a radial kernel 
  $$F(x,y)=\sum_{z\in \cX}f(d(x,z))f(d(z,y)),
  $$ 
cf. \eqref{eq:conv}, 
and the corresponding version of the invariance principle expresses $D$ as a difference between the average of $F$ over ${\cX}$ 
and the average over $Z.$ This approach can start with either $f$ or $F$, where each of the options has its 
own benefits. On the one hand, the function $F$ corresponds to a potential on ${\cX}$ (as in \eqref{eq:f}) and is related to the
geometric nature of the problem. At the same time, it is not always easy to find $f$ given $F$, although it is $f$ that is required
to define the discrepancy. 
On the other hand, starting with $f$, we can express
$F$ as a convolution, which implies that $F$ is a positive-definite radial kernel on ${\cX}.$ 
This enables one to study the problem of discrepancy minimization by linear programming as discussed in the previous section in a special case.

\subsection{Metric association schemes} Suppose that the finite set ${\cX}$ supports the structure of a commutative association scheme with $d$ classes $R_1,\dots,R_d$ whereby the distance $d(x,y)$ is replaced by the value $j$ such that $(x,y)\in R_j.$ It is easy to extend the definition
of discrepancy and prove a corresponding version of the invariance principle to this setting. We briefly discuss one special case which 
enables one to simplify the invariance principle \eqref{eq:SG}. Namely, suppose that the association scheme on ${\cX}\times{\cX}$ is {\em metric}, i.e., 
the pair $({\cX},R_1)$ where $R_1=:\{(x,y)\in {\cX}^2|d(x,y)=1\}$
forms a distance-regular graph \cite{bro89}. Let $R_j:=\{(x,y)|d(x,y)=j\}, j=1,\dots,n$ be the set of pairs that have distance $j$ in the graph, where $n$ is the diameter of the graph, and let $R_0:=\{(x,x)|x\in {\cX}\}$. The defining property of the association scheme is as follows: the number
    \begin{equation}\label{eq:int}
    p_{ij}(x,y)=|\{u\in {\cX}|d(x,u)=i,d(y,u)=j\}|
    \end{equation}
depends only on the distance $d(x,y).$ If $d(x,y)=k,$ then this number is denoted by $p_{ij}^k$ and is called the intersection 
number of the scheme. In particular, the numbers $n_i:=p_{ii}^0=|\{u\in {\cX}|d(u,x)=i\}|$, which do not depend on $x$, are called the valencies of the 
scheme. For a metric association scheme $A({\cX},R_0,R_1,\dots,R_m)$ the numbers $p_{ij}^k\ne 0$ only if $|i-j|\le k\le i+j,$ which guarantees that
the triangle inequalities are satisfied. A large number of metric association schemes are known in the literature \cite{bro89}. 

For metric schemes, the general Stolarsky identity \eqref{eq:SG} admits certain simplifications.
\begin{theorem} Let $Z\subset {\cX},|Z|=N$ be a code in a metric association scheme ${\cX}$ of diameter $n$, and let 
$A_k=\frac 1N|\{(z_1,z_2)\in Z^2\mid (z_1,z_2)\in R_k\}|, k=1,\dots,n.$ Then
    \begin{equation}\label{eq:metric}
    D^{L_2}(Z)=\frac12\Big(\frac 1{|{\cX}|}\sum_{i,j=0}^n\sum_{k=|i-j|}^{i+j}n_kp_{ij}^k|i-j|-\frac1{N}\sum_{k=1}^n A_k\sum_{\begin{substack}{i,j\\|i-j|\le k\le i+j}\end{substack}} p_{ij}^k |i-j|\Big).
    \end{equation}
\end{theorem}
\begin{proof}
The proof amounts to a direct calculation:
  \begin{align*}
   \sum_{x,y\in{\cX}}\sum_{u\in{\cX}}|d(x,u)-d(y,u)|&=\sum_{x\in{\cX}}\sum_{k=0}^n\sum_{y:(x,y)\in R_k}\sum_{\begin{substack}{i,j\\|i-j|\le k\le i+j}\end{substack}}
     \sum_{\begin{substack}{u\in {\cX}\\(u,x)\in R_i,(u,y)\in R_j}\end{substack}}|i-j|\\
     &=\sum_{x\in{\cX}}\sum_{k=0}^n \sum_{y:(x,y)\in R_k}\sum_{i,j=0}^n p_{ij}^k|i-j|\\
     &=\sum_{x\in{\cX}}\sum_{i,j=0}^n\sum_{k=|i-j|}^{i+j} n_k p_{ij}^k|i-j|.
   \end{align*}
 The sums on $i,j,k$ on the last line do not depend on $x,$ which proves the equality
       \begin{equation*}
   \frac1{|{\cX}|^2}\sum_{x,y\in{\cX}}\sum_{u\in{\cX}}|d(x,u)-d(y,u)|=\frac 1{|{\cX}|}\sum_{i,j=0}^n\sum_{k=|i-j|}^{i+j}n_kp_{ij}^k|i-j|.
     \end{equation*}
Suppose that $d(x,y)=k.$ As a part of the above calculation we have shown that
   \begin{gather*}
   \sum_{u\in {\cX}}|d(x,u)-d(y,u)|=\sum_{\begin{substack}{i,j\\|i-j|\le k\le i+j}\end{substack}} p_{ij}^k |i-j|.
   \end{gather*}
 Using these results in \eqref{eq:SG} finishes the proof.
\end{proof}   
   
By definition, the kernel $\mu_t$  considered above in \eqref{eq:mut}, can be expressed via the intersection numbers as follows:
$\mu_t(x,y)=\sum_{i,j=0}^t p_{ij}^w,$ where $w=d(x,y).$ This expression as well as \eqref{eq:metric} apply to all metric schemes; however, they do
not always lead to a simplified evaluation of $D^{L_2}(Z)$. For instance, in the calculations for the Hamming space we found it more convenient to argue from the first principles.

\subsection{Other related problems} 
Among natural extensions of the problems considered in this paper we mention working out the details of the invariance principle in
the nonbinary Hamming space and in the Johnson space (the set of all binary $n$-vectors of a fixed Hamming weight). While the former
likely is similar to the binary case, the latter may reveal interesting combinatorial connections. As another problem of interest we 
mention studying $L_p$ discrepancies for finite spaces, as has been recently done for the spherical case in \cite{skriganov2018bounds}. Bounds
for $L_p$ discrepancies in the Hamming space were obtained in a follow-up work \cite{Barg2020a}.

While we found some examples of discrepancy-minimizing configurations in the Hamming space, we did not attempt to exhaust or classify
codes that minimize quadratic discrepancy. This question in general does not look easy, and we do not have intuitively appealing
conjectures regarding the minimizers. A related problem is to further study relations between codes that have the smallest sum of
distances and codes with the smallest discrepancy in the Hamming space.

It is also possible to replace metric balls in the space with other geometric shapes in the definition of discrepancy, establishing a corresponding version of invariance. For the spherical case this leads to interesting results when spherical caps of all radii are replaced by the hemispheres or spherical wedges \cite{BilykDai2019,Bilyk2018,Skriganov2017}. Bounds on discrepancy for hemispheres in the Hamming space were recently derived in
\cite{Barg2020a}, while the question of identifying other shapes of interest is still open.

Studies of the invariance principle in the spherical case also involve interesting analytic problems which are to an extent absent in the finite case. For instance, absolute bounds on discrepancy arise by assuming that the set $Z$ forms a spherical design of strength $t$, meaning that the first $t$ moments of $Z$ are zero. Estimating the tail of the Fourier expansion of discrepancy then enables one to bound $D^{L_2}(Z)$ for 
designs, and these bounds asymptotically agree with the classical lower bounds for certain classes of designs; see \cite{Bilyk2018,Brauchart2019,Skriganov2019} for details. Finally, it is possible to replace finite point sets with general distributions on the sphere \cite{Bilyk2018}. In the Hamming space configurations that form $t$-designs (orthogonal arrays) do not necessarily have small discrepancy, and asymptotic bounds necessarily involve increasing the dimension $n$ (unlike the case of $S^d$ where $d$ is fixed while the size of the set $N$ increases). A question of interest for the Hamming space is to quantify gap to the uniform distribution for different classes of codes. In particular, many code families
come with known bounds on the minimum or maximum distance, and this information could be useful for computing bounds on their discrepancy. Earlier
results of this kind for energy minimization were recently obtained in \cite{BDHSS2019a}.

{\sc Acknowledgments.} I am grateful to Maxim Skriganov for detailed comments on the first version of this paper, and to Peter Boyvalenkov and Patrick Sol{\'e} for useful discussions.
Computations were aided by the Mathematica package {\small \tt fastZeil} which is a part of the {\small\tt RISCErgoSum} bundle written by the group of Prof. Peter Paule, University of Linz, Austria \cite{RISC}. 

This research was partially supported by NSF grants CCF1618603 and CCF1814487.

\section*{Appendix: A simple proof of \eqref{eq:SoS}}

This proof was found after the paper appeared in print. We will show that  
  \begin{equation}\label{eq:37'}
     \sum_{k=0}^n (K_k^{(n)}(i))^2=(-1)^i K_n^{(2n)}(2i)=\binom{2n-2i}{n-i}\binom{2i}i\Big/\binom ni. 
   \end{equation}
Using the generating function \eqref{eq:gf} we obtain
\begin{align*}
\sum_{k=0}^n K_k^{(n)}(w)z^k&\sum_{k'=0}^n K_{k'}^{(n)}(w)z^{-k'}=(1+z)^{n-w}(1-z)^w (1+z^{-1})^{n-w} (1-z^{-1})^w\\
  &=(-1)^w z^{-n}(1+z)^{2(n-w)}(1-z)^{2w}\\
   &=(-1)^w z^{-n}\sum_{k=0}^{2n} K_k^{(2n)}(2w)z^k.
   \end{align*}
Equating the constant terms on both sides, we obtain the first equality in \eqref{eq:37'}. Next, again using \eqref{eq:gf}, we have
  $$
  \sum_{k=0}^{2n} K_k^{(2n)}(n)z^k=(1-z^2)^n,
  $$
and thus,
  $K_{2k}^{(2n)}(n)=(-1)^k\binom nk.$ Now the symmetry relation \eqref{eq:symmetry} implies that 
  $\binom {2n}nK_{2k}^{(2n)}(n)=\binom{2n}{2k}K_n^{(2n)}(2k),$  and we get
  $$
  K_{2k}^{(2n)}(n)=\frac{\binom{2n}{2k}}{\binom {2n}{n}}K_n^{(2n)}(2k)= (-1)^k\binom nk
  $$
  or
  $$
  K_n^{(2n)}(2i)=(-1)^i\frac{\binom in\binom{2n}n}{\binom{2n}{2i}}.
  $$
Rewriting the binomial coefficients on the right, we obtain the second equality in \eqref{eq:37'}.

	\bibliographystyle{IEEEtranS}
	\bibliography{./stolarsky}
\end{document}